\documentclass[reqno,12pt]{amsart} 
\usepackage{amsmath}
\usepackage{amssymb}
\usepackage{amsthm}
\usepackage{xcolor}
\newcommand\NoBlackBoxes{\global\overfullrule0pt}
\NoBlackBoxes
\theoremstyle{plain} 
\newtheorem{theorem}{Theorem} 
\newtheorem{lemma}[theorem]{Lemma}
\newtheorem{proposition}[theorem]{Proposition}

\newtheorem{corollary}[theorem]{Corollary}
\def\4{\kern1pt}

\def\6{\vphantom0}

\def\8{\kern-10pt}
\def\7#1{_{(#1)}}

\theoremstyle{definition}
\newtheorem*{assumption (H1)}{Assumption (H1)}
\newtheorem*{assumption (H2)}{Assumption (H2)}

\theoremstyle{remark}
\newtheorem{remark}[theorem]{Remark}

\numberwithin{equation}{section}
\numberwithin{theorem}{section}

\makeatletter
\let\serieslogo@\relax
\let\@setcopyright\relax

\def\speciallabelmark#1{\def\@currentlabel{#1}}
\makeatother

\newcommand{\Cal}{\mathcal}

\newcommand{\E}{{\mathbb E}}

\newcommand{\beqa}{\begin{eqnarray}}
\newcommand{\beqan}{\begin{eqnarray*}}
\newcommand{\eeqa}{\end{eqnarray}}
\newcommand{\eeqan}{\end{eqnarray*}}
\def\beq#1\eeq{\begin{equation}#1\end{equation}}

\begin{document}

\def\ffrac#1#2{\raise.5pt\hbox{\small$\4\displaystyle\frac{\,#1\,}{\,#2\,}\4$}}
\def\ovln#1{\,{\overline{\!#1}}}
\def\ve{\varepsilon}
\def\kar{\beta_r}

\title{Edgeworth-type expansion in the entropic free CLT}

\author{G. P. Chistyakov$^{1,2}$}
\thanks{1) Faculty of Mathematics , 
University of Bielefeld, Germany.} 
\thanks{2) Research supported by SFB 701.}
\address
{Gennadii Chistyakov \newline
Fakult\"at f\"ur Mathematik\newline
Universit\"at Bielefeld\newline
Postfach 100131\newline
33501 Bielefeld \newline
Germany}
\email {chistyak@math.uni-bielefeld.de} 

\author{F. G\"otze$^{1,2}$}
\address
{Friedrich G\"otze\newline
Fakult\"at f\"ur Mathematik\newline
Universit\"at Bielefeld\newline
Postfach 100131\newline
33501 Bielefeld \newline
Germany}
\email {goetze@math.uni-bielefeld.de}

\date{January, 2017}


\maketitle
\markboth{ G. P. Chistyakov and F. G\"otze}{Rate of convergence}

\begin{abstract}
We prove an expansion for densities in the free CLT and apply this result to an expansion
in the entropic free central limit theorem assuming a~moment condition of order four for the free summands.  
\end{abstract}

\section{Introduction}
Free convolutions were introduced
by D. Voiculescu~\cite{Vo:1986}, \cite{Vo:1987} and have been studied 
intensively in context of non commutative probability.
The~key concept here is the~notion of freeness,
which can be interpreted as a~kind of independence for
non commutative random variables. As in~classical probability theory where
the~concept of independence gives rise to the~classical convolution, 
the~concept of freeness leads to a~binary operation on the~probability measures, 
the~free convolution. Many classical results 
in the~theory of addition of independent random variables have their
counterparts in Free Probability, such as the~Law of Large Numbers,
the~Central Limit Theorem, the~L\'evy-Khintchine formula and others.
We refer to Voiculescu, Dykema and Nica \cite{Vo:1992}, Hiai 
and Petz~\cite{HiPe:2000}, and Nica and Speicher \cite{NSp:2006} for an~introduction 
to these topics.

In this paper we 
obtain an~analogue of Esseen's expansion for a~density of normalized sums of free identically distributed
random variables under a fourth moment assumption on the free summands. Using this expansion we establish
the rate of convergence of the free entropy of normalized sums of free identically distributed
random variables.

The~paper is organized as follows. In Section~2 we formulate and
discuss the main results of the~paper. In Section~3 and 4 we state 
auxiliary results. 
In Section~5 we discuss the passage to probability measures with bounded supports.
In Section~6 we obtain a~local asymptotic expansion for a~density
in the~CLT for free identically distributed random variables. 
In Section~7 we study the behaviour of subordination functions in the free CLT for truncated free summands.
In Section~8 we discuss the closeness of subordination functions in the free CLT for bounded and unbounded  free
random variables.
In Section~9 we  investigate the rate of convergence for densities in the free CLT in $L_1(-\infty,+\infty)$ and 
Section~10 is devoted to study the rate of convergence for the free entropy of normalized sums of free identically distributed
random variables.
In Section~11 we derive rates of convergence for the free Fisher information of normalized sums of free identically distributed
random variables.

\section{Results}
Denote by $\mathcal M$ the~family of all Borel probability measures
defined on the~real line $\mathbb R$.  
Let $\mu\boxplus\nu$ be
the~free (additive) convolution of $\mu$ and $\nu$ introduced by 
Voiculescu~\cite{Vo:1986} for compactly supported measures. 
Free convolution was extended by
Maassen~\cite{Ma:1992} to measures with finite variance and by Bercovici
and Voiculescu~\cite{BeVo:1993} to the~class $\mathcal M$.
Thus, $\mu\boxplus\nu=\mathcal L(X+Y)$,
where $X$ and $Y$ are free random variables such that $\mu=\mathcal L(X)$ and $\nu=\mathcal L(Y)$.

Henceforth $X,X_1,X_2,\dots$ stands for a~sequence of identically distributed random variables with
distribution $\mu=\mathcal L(X)$. Define 
$m_k(\mu):=\int_{\mathbb R}u^k\,\mu(du)$,
where $k=0,1,\dots$.

The~classical CLT says that if $X_1,X_2,\dots$ are independent
and identically distributed random variables with a~probability 
distribution $\mu$ such that $m_1(\mu)=0$ and $m_2(\mu)=1$, 
then the~distribution function $F_n(x)$ of 
\begin{equation}\label{2.1}
Y_n:=\frac{X_1+X_2+\dots +X_n}{\sqrt n}
\end{equation}
tends to the~standard Gaussian law $\Phi(x)$ as $n\to\infty$ uniformly
in $x$.

A~free analogue of this classical result was proved by Voiculescu~\cite{Vo:1985} for bounded
free random variables and 
later generalized by Maassen~\cite {Ma:1992} to unbounded random variables.
Other generalizations can be found in \cite{BeVo:1995}, \cite{BeP:1999}, \cite{ChG:2005a}, 
\cite {Ka:2007}--\cite {Ka:2007b}, \cite{P:1996}, \cite{Vo:2000}, \cite{Wa:2010}. 

For $t>0$, the centered semicircle distribution of variance $t$ is the probability measure with density
$p_{w_t}(x):=\frac 1{2\pi t}\sqrt{(4t-x^2)_+}, \,x\in\mathbb R$, where $a_+:=\max\{a,0\}$ for $a\in\mathbb R$. 
Denote by $\mu_{w_t}$ the~probability measure with the~distribution function $w_t(x)$. In the sequel we use the notations
$w_1(x)=w(x)$.

When the~assumption of independence is 
replaced by the~freeness of the~non commutative random variables
$X_1,X_2,\dots,X_n$, the~limit distribution function of (\ref{2.1}) is 
the~semicircle law $w(x)$.  
We denote as well by $\mu_n$ the probability measure with
the distribution function $F_n(x)$.

It was proved in~\cite{BelBer:2004} that if the distribution $\mu$ of $X$ is not a Dirac measure,
then in the free case $F_n(x)$ is Lebesgue absolutely continuous when $n\ge n_1=n_1(\mu)$ is sufficiently large. 
Denote by $p_n(x)$ the density of $F_n(x)$.

In the sequel we denote by $c(\mu),c_1(\mu),c_2(\mu),\dots$
positive constants depending on $\mu$ only. By $c(\mu)$ we denote generic constants in different 
(or even in the same) formula. The symbols $c_1(\mu),c_2(\mu),\dots$ will denote explicit constants.
By $\{\varepsilon_{nk}\}$ denote positive numbers such that $\varepsilon_{nk}\to 0$ as $n\to\infty$.

Wang~\cite{Wa:2010} proved that under the condition $m_2(\mu)<\infty$ 
the density $p_n(x)$ of $F_n(x)$ is continuous for sufficiently large $n$ and  
\begin{equation}\label{asden3}
p_n(x)\le c(\mu),\quad x\in\mathbb R.
\end{equation}

Assume that $m_4(\mu)<\infty,m_1(\mu)=0,m_2(\mu)=1$ and denote 
\begin{equation}\label{2.3****}
a_n:=\frac{m_3(\mu)}{\sqrt n},\quad b_n:=\frac{m_4(\mu)-m_3^2(\mu)-1}n,\quad
d_n:=\frac{m_4(\mu)-m_3^2(\mu)}n,\quad n\in\mathbb N.
\end{equation}
Furthermore, let  $e_n:=(1-b_n)/\sqrt{1-d_n}$ and let $I_n$ and $I_n^*$ denote intervals of the form
\begin{equation}\label{asden1a}
I_n:=\Big\{x\in\mathbb R:|x-a_n|\le \frac 2{e_n}-\frac{\varepsilon_{n1}}n\Big\},\qquad
I_n^*:=\Big\{x\in\mathbb R:|x-a_n|\le \frac 2{e_n}-\sqrt{\frac{\varepsilon_{n1}}{n}}\Big\}.
\end{equation}
In the sequel 
we denote by $\theta$ a real-valued quantity such that $|\theta|\le 1$.

We have derived an asymptotic expansion of $p_n(x)$ for bounded 
free random variables $X_1,X_2,\dots$ in the paper~\cite{ChG:2011}. Improving the methods of this paper and \cite{ChG:2011a} we
obtain an asymptotic expansion of $p_n(x)$ for the case $m_4(\mu)<\infty$. Denote by $v_n(x)$ the function
\begin{equation}\label{vden}
 v_n(x)=\Big(1+\frac 12 d_n-a_n^2-\frac 1{n}-a_nx-\Big(b_n-a_n^2-\frac 1{n}\Big)x^2\Big)p_w(e_nx),\quad x\in\mathbb R.
\end{equation}

\begin{theorem}\label{th7}
Let $m_4(\mu)<\infty$ and $m_1(\mu)=0,\,m_2(\mu)=1$. Then there exist sequences $\{\varepsilon_{n1}\}$ and $\{\varepsilon_{n2}\}$
such that
\begin{equation}\label{asden}
p_n(x+a_n)=
v_n(x)+\rho_{n1}(x)+\rho_{n2}(x),\quad x\in I_n^*-a_n,
\end{equation}
where, for $ x\in I_n^*-a_n$,
\begin{equation}
|\rho_{n1}(x)|\le\frac {\varepsilon_{n1}}{n}\frac{c(\mu)}{(4-(e_nx)^2)^{3/2}},\label{asden1} \\
\end{equation}
and for $ x\in (I_n-a_n)\setminus(I_n^*-a_n)$,
\begin{equation}
|\rho_{n1}(x)|\le \sqrt{\frac {\varepsilon_{n1}}{n}}\frac {c(\mu)}{(4-(e_nx)^2)^{1/2}}.\label{asden1*}
\end{equation}
In $(\ref{asden})$
$\rho_{n2}(x)$ is a continuous function such that
\begin{equation}
0\le\rho_{n2}(x)\le c(\mu)\quad\text{ and}\quad \int_{I_n-a_n}\rho_{n2}(x)\,dx=o(1/n^{2}).\label{asden1**}
\end{equation}
Moreover,
\begin{equation}\label{asden2}
\int_{\mathbb R\setminus I_n}p_n(x)\,dx\le \frac {\varepsilon_{n2}}n. 
\end{equation}
\end{theorem}

\begin{corollary}\label{corth7.1}
Let $m_4(\mu)<\infty$ and $m_1(\mu)=0,\,m_2(\mu)=1$, 
then
\begin{equation}\label{2.9}
\int_{\mathbb R}|p_n(x)-p_w(x)|\,dx=\frac {2|m_3(\mu)|}{\pi\sqrt n}+c(\mu)\theta \Big(\Big(\frac{\varepsilon_{n1}}n\Big)^{3/4}+\frac{1}n\Big).
\end{equation}
\end{corollary}

In \cite{ChG:2011} we proved analogous results
for {\it bounded} free random variables and in \cite{ChG:2011a} 
assuming  a finite moment of order eight.

Recall that, if the random variable $X$ has density $f$, then the classical entropy of a
distribution of $X$ is defined as $h(X)=-\int_{\mathbb R}f(x)\log f(x)\,dx$,
provided the positive part of the integral is finite. Thus we have $h(X)\in[-\infty,\infty)$.
A much stronger statement than the classical CLT -- the entropic central limit theorem --
indicates that, if for some $n_0$, or equivalently, for all $n\ge n_0$, $Y_n$ from (\ref{2.1})
have absolutely continuous distributions with finite entropies $h(Y_n)$, then there is convergence
of the entropies, $h(Y_n)\to h(Y)$, as $n\to \infty$, where $Y$ is a standard Gaussian random variable. 
This theorem is due to Barron~\cite{Ba:1986}. Artstein, Bally, Barthez, and Naor~\cite{Ar:2004} have solved 
an old question raised by Shannon about the monotonicity of entropy under convolution. The relative entropy 
$$
D(X)=D(X||Z)=h(Z)-h(X),
$$
where the normal random variable $Z$ have the same mean and the same variance as $X$, is nonnegative  and 
serves as kind of a distance to the class of normal laws. Thus, the entropic central limit theorem may be reformulated
as $D(Y_n)\downarrow 0$, as long as $D(Y_{n_0})<+\infty$ for some $n_0$.

Recently Bobkov, Chistyakov and G\"otze~\cite{BChG:2013}
found the rate of convergence to zero of $D(Y_n)$ and for the random variables $X$ with $\E |X|^k<\infty, \,k\ge 4$,
have obtained an Edgeworth-type expansion of $D(Y_n)$ as $n\to\infty$.

Let $\nu$ be a probability measure on $\mathbb R$. We assume below that $m_1(\nu)=0$ and $m_2(\nu)=1$.
The quantity
\begin{equation}\notag
\chi(\nu)=\int\int_{\mathbb R\times \mathbb R}\log|x-y|\,\nu(dx)\nu(dy)+\frac 34+\frac 12\log 2\pi,
\end{equation}
called free entropy, was introduced by Voiculescu in~\cite{Vo:1993}. Free entropy $\chi$ behaves like 
the classical entropy $h$. 
In particular, the free entropy is maximized 
by the standard semicircular law $w$ with the value $\chi(w)=\frac 12\log 2\pi e$ among all 
probability measures with variance one~\cite{HiPe:2000}, \cite{Vo:1997}. 
Shlyakhtenko~\cite{Sh:2007}
has proved that $\chi(\mu_n)$ decreases monotonically, i.e., the Shannon hypothesis holds in the free case as well.

Wang~\cite{Wa:2010}
has proved a~free analogue of Barron's result: the free entropy $\chi(\mu_n)$ converges to the semicircular entropy.
As in the classical case a relative free entropy
$$
D(\nu||\mu_w)=\chi(\mu_w)-\chi(\nu)
$$
is nonnegative and serves as kind of a distance to the class of semicircular laws.

We derive an optimal rate of convergence in the free CLT
for free random variables with a~finite moment of order four. In previous results \cite{ChG:2011} we showed  
an analogous result for bounded free random variables and in \cite{ChG:2011a} for free random variables with a finite moment of order eight.
\begin{corollary}\label{corth7.2}
Let $m_4(\mu)<\infty$ and $m_1(\mu)=0,\,m_2(\mu)=1$. Then, for every fixed $1<q\le 1.01$,
\begin{align}
D(\mu_n||\mu_w)=\frac {m_3^2(\mu)}{6n}+\theta\Big(c(\mu,q)\Big(\frac{\varepsilon_{n1}}n\Big)^{1+\frac 1{2q}}
+c(\mu)\frac{\varepsilon_{n2}}n\Big),\label{co2} 
\end{align} 
where $c(\mu,q)>0$ is a constant depended on $\mu$ and $q$ only. 
\end{corollary}
Hence the remainder term in (\ref{co2}) is of order $o(n^{-1})$ provided that $m_4(\mu)<\infty$. In Sections~5 and 8 we 
explicitly describe the sequences $\{\varepsilon_{n1}\}$ and $\{\varepsilon_{n2}\}$. If we assume
that $m_6(\mu)<\infty$, then it follows from Remarks~\ref{rem8.11} and \ref{rem8.12} (see the end of Section 8)
that the remainder term in (\ref{co2}) is of order $O(n^{-3/2})$.

Given a random variable $X$ with an absolutely continuous density $f$, the Fisher information of $X$ is defined by
$I(X)=\int_{-\infty}^{+\infty}\frac{f'(x)^2}{f(x)}\,dx$, where $f'$ denotes the Radon-Nikodym derivative of $f$. In all other cases,
let $I(X)=+\infty$. With the first two moments of $X$ being fixed, $I(X)$ is minimized for the normal random variable $Z$ with 
the same mean and the same variance as $X$, i.e. $I(X)\ge I(Z)$ (which is a variant of Cram\'er-Rao's inequality).

Baron and Johnson have proved in \cite{BaJoh:2004} that $I(Y_n)\to I(Z)$, as $n\to\infty$, if and only if $I(Y_{n_0})<\infty$.
In  classical probability and statistics the relative Fisher information
$$
I(X||Z)=I(X)-I(Z)
$$
is used as a strong measure of the probability distribution of  $X$ being near to the Gaussian distribution. The result of Baron and Johnson is equivalent to the fact
that $I(Y_n||Z)\to 0$ as $n\to \infty$, if and only if $I(Y_{n_0}||Z)<\infty$.

Bobkov, Chistyakov and G\"otze~\cite{BChG:2014} found the rate of convergence to zero of $I(Y_n||Z)$ and for the random variables $X$ with $\E |X|^k<\infty, \,k\ge 4$,
have obtained an Edgeworth-type expansion of $I(Y_n||Z)$ as $n\to\infty$.

Suppose that the measure $\nu$ has a density $p$ in $L^3(\mathbb R)$. Then, following Voiculescu~\cite{Vo:1997},
the free Fisher information is
\begin{equation}\notag
\Phi(\nu)=\frac{4\pi^2}3\int_{\mathbb R}p(x)^3\,dx.
\end{equation}
It is well-known that $\Phi(\mu_w)=1$. The free Fisher information has many properties analogous to those
of classical Fisher information. These include the free analog of the Cram\'er-Rao inequality.

Assume now that $m_1(\nu)=0$ and $m_2(\nu)=1$. Consider 
the free relative Fisher information
$$
\Phi(\nu||\mu_w)=\Phi(\nu)-\Phi(\mu_w)\ge 0
$$
as a strong measure of closeness of $\nu$ to Wigner's semicircle law.
Here we obtain an Edgeworth-type expansion for free random variables with a finite moment of order four.
\begin{corollary}\label{corth7.3}
Let $m_4(\mu)<\infty$ and $m_1(\mu)=0,\,m_2(\mu)=1$. Then
\begin{equation}\label{2.11}
\Phi(\mu_n||\mu_{w})=\int_{\mathbb R}p_n(x)^3\,dx-\Phi(\mu_w)=\frac{m_3^2(\mu)}n+c(\mu)\theta\,\frac{\varepsilon_{n1}+\varepsilon_{n2}}n. 
\end{equation} 
\end{corollary}
As in the formula (\ref{co2}) the remainder term here is of order $o(n^{-1})$ if $m_4(\mu)<\infty$ and of order
 $O(n^{-3/2})$ provided that $m_6(\mu)<\infty$.

In contrast to the classical case (see \cite{BChG:2013} and \cite{BChG:2014}) we expect that the asymptotic expansion 
in (\ref{co2}) and (\ref{2.11}) holds with an error of order $n^{-1}$ only.

\section{Auxiliary results}

We need results about some classes of analytic functions
(see {\cite{Akh:1965}, Section~3.

The~class $\mathcal N$ (Nevanlinna, R.) is the~class of analytic 
functions $f(z):\mathbb C^+\to\{z: \,\Im z\ge 0\}$.
For such functions there is an~integral representation
\begin{equation}\label{3.1}
f(z)=a+bz+\int\limits_{\mathbb R}\frac{1+uz}{u-z}\,\tau(du)=
a+bz+\int\limits_{\mathbb R}\Big(\frac 1{u-z}-\frac u{1+u^2}\Big)(1+u^2)
\,\tau(du),\quad z\in\mathbb C^+,
\end{equation}
where $b\ge 0$, $a\in\mathbb R$, and $\tau$ is a~non-negative finite
measure. Moreover, $a=\Re f(i)$ and $\tau(\mathbb R)=\Im f(i)-b$.   
From this formula it follows that 
$f(z)=(b+o(1))z$
for $z\in\mathbb C^+$
such that $|\Re z|/\Im z$ stays bounded as $|z|$ tends to infinity (in other words
$z\to\infty$ non tangentially to $\mathbb R$).
Hence if $b\ne 0$, then $f$ has a~right inverse $f^{(-1)}$ defined
on the~region 
$
\Gamma_{\alpha,\beta}:=\{z\in\mathbb C^+:|\Re z|<\alpha \Im z,\,\Im z>\beta\}
$
for any $\alpha>0$ and some positive $\beta=\beta(f,\alpha)$.

A~function $f\in\mathcal N$ admits the~representation
\begin{equation}\label{3.3}
f(z)=\int\limits_{\mathbb R}\frac{\sigma(du)}{u-z},\quad z\in\mathbb C^+,
\end{equation}
where $\sigma$ is a~finite non-negative measure, if and only if
$\sup_{y\ge 1}|yf(iy)|<\infty$. Moreover $\sigma(\mathbb R)=-\lim_{y\to+\infty}iyf(iy)$.

For $\mu\in\mathcal M$, consider its Cauchy transform $G_{\mu}(z)$
\begin{equation}\label{3.5a}
G_{\mu}(z)=\int_{\mathbb R}\frac{\mu(du)}{z-u},\quad z\in\mathbb C^+. 
\end{equation}

The measure $\mu$ can be recovered from $G_{\mu}(z)$ as
the weak limit of the measures
\begin{equation}\notag
\mu_y(dx)=-\frac 1{\pi}\Im G_{\mu}(x+iy)\,dx,\quad x\in\mathbb R,\,\,y>0,
\end{equation}
as $y\downarrow 0$. If the function $\Im G_{\mu}(z)$ is continuous at $x\in\mathbb R$,
then the probability distribution function $D_{\mu}(t)=\mu((-\infty,t))$ is differentiable
at $x$ and its derivative is given by 
\begin{equation}\label{3.4}
D_{\mu}'(x)=-\Im G_{\mu}(x)/\pi. 
\end{equation}
This inversion formula
allows to extract the density function of the measure $\mu$ from its Cauchy transform.

Following Maassen~\cite{Ma:1992} and Bercovici and 
Voiculescu~\cite{BeVo:1993}, we shall consider in the~following
the~ {\it reciprocal Cauchy transform}
\begin{equation}\label{3.5}
F_{\mu}(z)=\frac 1{G_{\mu}(z)}.
\end{equation}
The~corresponding class of reciprocal Cauchy
transforms of all $\mu\in\mathcal M$ will be denoted by $\mathcal F$.
This class coincides with the~subclass of Nevanlinna functions $f$
for which $f(z)/z\to 1$ as $z\to\infty$ non tangentially to $\mathbb R$.

The~following lemma is well-known, see~\cite{Akh:1965}, Th. 3.2.1, p. 95. 
\begin{lemma}\label{3.4abl}
Let $\mu$ be a~probability measure such that
\begin{equation}\label{3.4abl1}
m_k=m_k(\mu):=\int\limits_{\Bbb R}u^k\,\mu(du)<\infty,\qquad k=0,1,\dots,2n,\quad n\ge 1.
\end{equation}
Then the~following relation holds
\begin{equation}\label{3.4abl2}
\lim_{z\to\infty}z^{2n+1}\Big(G_{\mu}(z)-\frac 1z-\frac{m_1}{z^2}-
\dots-\frac{m_{2n-1}}{z^{2n}}\Big)=m_{2n}
\end{equation}
uniformly in the~angle $\delta\le\arg z\le\pi-\delta$, 
where $0<\delta<\pi/2$.

Conversely, if for some function $G(z)\in\mathcal N$ the~relation $(\ref{3.4abl2})$
holds with real numbers $m_k$ for $z=iy,y\to\infty$, then $G(z)$ admits
the~representation~$(\ref{3.5a})$, where $\mu$ is a probability measure with moments $(\ref{3.4abl1})$.
\end{lemma}

As shown before, $F_{\mu}(z)$ admits the representation (\ref{3.1}) with $b=1$.
From Lemma~\ref{3.4abl} the following proposition is immediate.
\begin{proposition}\label{3.3^*pro} 
In order that a~probability measure $\mu$ satisfies the assumption $(\ref{3.4abl1})$, where $m_1(\mu)=0$, 
it is necessary and sufficient that
\begin{equation}\label{3.3^*proa}
F_{\mu}(z)=z+\int_{\mathbb R}\frac{\tau(du)}{u-z},\quad z\in\mathbb C^+, 
\end{equation}
where $\tau$ is a nonnegative measure such that $m_{2n-2}(\tau)<\infty$. Moreover
\begin{equation}\label{3.3^*prob}
m_k(\mu)=\sum_{l=1}^{[k/2]}\sum_{s_1+\dots+s_l=k-2,\,s_j\ge 0}m_{s_1}(\tau)\dots m_{s_l}(\tau),\quad k=2,\dots,2n. 
\end{equation}
\end{proposition}

Voiculescu~\cite{Vo:1993} showed for compactly supported probability measures that
there exist unique functions $Z_1, Z_2\in\mathcal F$ such that
$G_{\mu_1\boxplus\mu_2}(z)=G_{\mu_1}(Z_1(z))=G_{\mu_2}(Z_2(z))$ 
for all $z\in\mathbb C^+$.
Using Speicher's combinatorial approach~\cite{Sp:1998} to freeness,
Biane~\cite{Bi:1998} proved this result in the~general case.

Bercovici and Belinschi~\cite{BelBer:2007}, 
Belinschi~\cite{Bel:2008}, Chistyakov and G\"otze \cite {ChG:2005},
proved, using complex analytic methods, that
there exist unique functions $Z_1(z)$ and $Z_2(z)$ in the~class 
$\mathcal F$ such that, for $z\in\mathbb C^+$, 
\begin{equation}\label{3.9}
z=Z_1(z)+Z_2(z)-F_{\mu_1}(Z_1(z))\quad\text{and}\quad
F_{\mu_1}(Z_1(z))=F_{\mu_2}(Z_2(z)). 
\end{equation}
The~function $F_{\mu_1}(Z_1(z))$ belongs again to the~class $\mathcal F$ and 
there exists 
$\mu\in\mathcal M$ such that
$F_{\mu_1}(Z_1(z)) =F_{\mu}(z)$, where $F_{\mu}(z)=1/G_{\mu}(z)$ and 
$G_{\mu}(z)$ is the~Cauchy transform as in (\ref{3.5a}).  
The~measure $\mu$ depends on $\mu_1$ and $\mu_2$ only and $\mu=\mu_1\boxplus\mu_2$.

Specializing to $\mu_1=\mu_2=\dots=\mu_n=\mu$ write $\mu_1\boxplus\dots\boxplus\mu_n=
\mu^{n\boxplus}$.
The~relation (\ref{3.9}) admits the~following
consequence (see for example \cite{ChG:2005}, Section 2, Corollary 2.3).

\begin{proposition}\label{3.3pro}
Let $\mu\in\mathcal M$. There exists a~unique function $Z\in\mathcal F$ 
such that
\begin{equation}\label{3.10}
z=nZ(z)-(n-1)F_{\mu}(Z(z)),\quad z\in\mathbb C^+,
\end{equation}
and $F_{\mu^{n\boxplus}}(z)=F_{\mu}(Z(z))$.
\end{proposition}

Using the~representation (\ref{3.1}) for $F_{\mu}(z)$ we obtain
\begin{equation}\label{7.8}
F_{\mu}(z)=z+\Re F_{\mu}(i)+\int\limits_{\mathbb R}\frac {(1+uz)\,\tau(du)}{u-z},
\quad z\in\mathbb C^+,
\end{equation}
where $\tau$ is a~nonnegative measure such that $\tau(\mathbb R)=\Im F_{\mu}(i)-1$.
Denote $z=x+iy$, where $x,y\in\mathbb R$. We see that, for $\Im z>0$, 
\begin{equation}\notag
\Im \Big(nz-(n-1)F_{\mu}(z)\Big)=y\Big(1-(n-1)I_{\mu}(x,y)\Big),\quad\text{where}\quad                                                                                                                                                                                                                                                                                                                                                                                                                                                                                                                                                                                                                             
I_{\mu}(x,y):=\int\limits_{\mathbb R}\frac{(1+u^2)\,\tau(du)}{(u-x)^2+y^2}.
\end{equation}
For every real fixed $x$, consider the~equation
\begin{equation}\label{7.9}
 y\Big(1-(n-1)I_{\mu}(x,y)\Big)=0,\quad y>0.
\end{equation}
Since $y\mapsto I_{\mu}(x,y),\,y>0$, is positive and monotone, and decreases to $0$ as $y\to\infty$,
it is clear that the~equation (\ref{7.9}) has at most one positive solution. 
If such a~solution exists, denote it
by $y_n(x)$.
Note that (\ref{7.9}) does not have a~solution $y>0$ for any given $x\in\mathbb R$ if and only if
$I_{\mu}(x,0)\le 1/(n-1)$.
Consider the~set $S:=\{x\in\mathbb R:I_{\mu}(x,0)\le 1/(n-1)\}$. We put $y_n(x)=0$ for $x\in S$. 
We proved in~\cite{ChG:2011}, Section~3, p.13, that
the~curve $\gamma_n$ given by the~equation $z=x+iy_n(x),\,x\in\mathbb R$, is continuous and simple.

Consider the~open domain $\tilde{D}_n:=\{z=x+iy,\,x,y\in\mathbb R: y>y_n(x)\}$.
\begin{lemma}\label{l7.4}                                                                                                                
Let $Z\in\mathcal F$ be the~solution of the~equation $(\ref{3.10})$. The function $Z(z)$ maps $\mathbb C^+$
conformally onto $\tilde{D}_n$.  
Moreover the~function $Z(z),\,z\in\mathbb C^+$, 
is continuous up to the~real axis and it establishes a homeomorphism between the real axis and 
the~curve $\gamma_n$.                                                                                                                                      
\end{lemma}
This lemma was proved in~\cite{ChG:2011} (see Lemma~3.4). The following lemma was proved as well 
in~\cite{ChG:2011} (see Lemma~3.5).

\begin{lemma}\label{l7.5}
Let $\mu$ be a~probability measure such that $m_1(\mu)=0,m_2(\mu)=1$. Assume that 
$\int_{|u|>\sqrt{(n-1)/8}}u^2\,\mu(du)\le 1/10$ for some
positive integer $n\ge 10^3$.
Then the~following inequality holds
\begin{equation}\label{7.11*}
|Z(z)|\ge \sqrt{(n-1)/8},\qquad z\in\mathbb C^+, 
\end{equation} 
where $Z\in\mathcal F$ is the~solution of the~equation $(\ref{3.10})$.
\end{lemma}

The next lemma was proved in \cite{Ma:1992} and \cite{Wa:2010}.
\begin{lemma}\label{l7.6}
There exists a unique probability measure $\nu$ such that such that
$F_{\mu}(z)=z-G_{\nu}(z),\,z\in\mathbb C^+$, and, for every $n\ge 1$,
$F_{\mu_n}(z)=z-G_{\nu_{n-1}\boxplus w_t}(z),\,z\in\mathbb C^+\cup\mathbb R$,
where the measure $\nu_{n-1}$ is given by $d\nu_{n-1}(x)=d\nu(\sqrt n x)$ and $t=t(n)=(n-1)/n$.
\end{lemma}

Biane~\cite{Bi:1998} gave the following bound.
\begin{lemma}\label{l7.7}
Fix $t>0$ and the probability measure $\nu$. Then $|G_{\nu \boxplus w_t}(z)|\le t^{-1/2},\,z \in {\mathbb C}^{+}\cup \mathbb R$. 
\end{lemma}

\section{Free Meixner measures}

Consider the~three-parameter family of probability measures $\{\mu_{a,b,d}:
a\in\mathbb R, b<1, d<1\}$ with the~reciprocal Cauchy transform
\begin{equation}\label{2.3h}
\frac 1{G_{\mu_{a,b,d}}(z)}=a+\frac 12\Big((1+b)(z-a)+\sqrt{(1-b)^2(z-a)^2-4(1-d)}\Big),
\quad z\in\mathbb C,
\end{equation}
which we will call the~free centered (i.e. with mean zero) Meixner measures. 
In this formula we choose the~branch of the~square root determined by the~condition
$\Im z>0$ implies $\Im (1/G_{\mu_{a,b,d}}(z))\ge 0$.
These measures are counterparts of the~classical measures discovered by Meixner~\cite{Me:1934}.
The~free Meixner type measures occurred in many places in the~literature, see for example 
\cite{BoBr:2006}, \cite{SaYo:2001}.

Saitoh and Yoshida~\cite{SaYo:2001} have proved that
the~absolutely continuous part of the~free Meixner measure $\mu_{a,b,d},a\in\mathbb R,b<1,d<1$, is 
given by
\begin{equation}\label{2.3g}
\frac{\sqrt{4(1-d)-(1-b)^2(x-a)^2}}{2\pi f(x)},
\end{equation}
when $a-2\sqrt{1-d}/(1-b)\le x\le a+2\sqrt{1-d}/(1-b)$, where 
\begin{equation}\notag
f(x):=bx^2+a(1-b)x+1-d;  
\end{equation}

Saitoh and Yoshida proved as well that for $0\le b<1$ the (centered) 
free Meixner measure
$\mu_{a,b,d}$ is $\boxplus$-infinitely divisible.

As we have shown in~\cite{ChG:2011}, Section 4, it follows from Saitoh and Yoshida's results that the~probability measure
$\mu_{a_n,b_n,d_n}$ with the parameters $a_n,b_n,d_n$ from (\ref{2.3****}) is $\boxplus$-infinitely divisible 
and it is absolutely continuous with a~density of the~form (\ref{2.3g}) where
$a=a_n, b=b_n,d=d_n$ for sufficiently large $n\ge n_1(\mu)$.

\section{Passage to measures with bounded supports}

Let us assume that $\mu\in\mathcal M$ and $m_4(\mu)<\infty$. In addition let $m_1(\mu)=0$ and $m_2(\mu)=1$.
By Proposition~\ref{3.3pro}, there exists $Z(z)\in\mathcal F$
such that (\ref{3.10}) holds, and $F_{\mu^{n\boxplus}}(z)=F_{\mu}(Z(z))$.
Hence $F_{\mu_n}(z)=F_{\mu}(\sqrt n S_n(z))/\sqrt n,\,z\in\mathbb C^+$, 
where $S_n(z):=Z(\sqrt n z)/\sqrt n$. Since $m_1(\mu)=0,\,m_2(\mu)=1$ and $m_4(\mu)<\infty$, 
by Proposition~\ref{3.3^*pro},
we have the representation
\begin{equation}\label{Pas5.0a}
F_{\mu}(z)=z+\int_{\mathbb R}\frac{\tau(du)}{u-z},\quad z\in\mathbb C^+, 
\end{equation}
where $\tau$ is a nonnegative measure such that $\tau(\mathbb R)=1$ and $m_2(\tau)<\infty$. 

Denote, for $n\in\mathbb N$,
\begin{equation}
\eta(n;\tau):=inf_{0<\varepsilon\le 10^{-1/2}}g_n(\varepsilon;\tau),\quad\text{where}\quad
g_n(\varepsilon;\tau)=\varepsilon+\frac 1{m_2(\tau)\varepsilon^2}\int_{|u|>\varepsilon\sqrt{n-1}}u^2\,\tau(du).\notag
\end{equation}
It is easy to see that $0<\eta(n;\tau)\le 11$ and $\eta(n;\tau)\to 0$ monotonically as $n\to\infty$. 
Let $\delta_n\in (0,10^{-1/2}]$ be a point at which the infimum of the function $g_n(\varepsilon;\tau)$ is attained.
This means that
\begin{equation}\label{Pas5.0}
\eta(n;\tau)=\delta_n+ \frac 1{m_2(\tau)\delta_n^2}\int_{|u|>\delta_n\sqrt{n-1}}u^2\,\tau(du).
\end{equation}

Consider a function
\begin{equation}\label{Pas5.0b}
F(z)=z+\int_{\mathbb R}\frac{\tau^*(du)}{u-z}:=\int_{|u|\le \delta_n\sqrt{n-1}}\frac{\tau(du)}{u-z},\quad z\in\mathbb C^+.
\end{equation}
This function belongs to the class $\Cal F$ and therefore there exists the probability measure $\mu^*$
such that $F_{\mu^*}(z)=F(z),\,z\in\mathbb R$. The probability measure $\mu^*$ of course depends on
$n$. Moreover we conclude from the inversion formula that $\mu^*([-\frac{\sqrt{10}}3\delta_n\sqrt{n-1},\frac{\sqrt{10}}3\delta_n\sqrt{n-1}])=1$
for $n\ge n_1(\mu)$. Hence it follows that the support of $\mu^*$ is contained in the interval $[-\frac 13\sqrt{n-1},\frac 13\sqrt{n-1}]$.
By Proposition~\ref{3.3^*pro}, we see as well that $m_1(\mu^*)=0$ and
\begin{align}\label{Pas5.1}
m_2(\mu)-m_2(\mu^*)&=
\tau(\mathbb R\setminus[-\delta_n\sqrt{n-1},\delta_n\sqrt{n-1}])\notag\\
&\le\frac 1{\delta_n^2(n-1)}\int_{|u|>\delta_n\sqrt{n-1}}u^2\,\tau(du)\le c(\mu)\frac{\eta(n;\tau)}{n-1}.
\end{align}
Moreover
\begin{align}\label{Pas5.1a}
|m_3(\mu)-m_3(\mu^*)|&=|m_1(\tau)m_0(\tau)-m_1(\tau^*)m_0(\tau^*)|\notag\\
&\le |m_1(\tau)-m_1(\tau^*)|+|m_2(\mu)-m_2(\mu^*)||m_1(\tau^*)|\notag\\
&\le \int_{|u|>\delta_n\sqrt{n-1}}|u|\,\tau(du)+c(\mu)|m_1(\tau^*)|\frac{\eta(n;\tau)}{n-1}\le c(\mu)\frac{\eta(n;\tau)}{\sqrt{n-1}};\notag\\
\end{align}
In the same way
\begin{align}\label{Pas5.1b}
|m_4(\mu)-m_4(\mu^*)|&\le |m_2(\tau)-m_2(\tau^*)|+|m_2(\mu)-m_2(\mu^*)||m_2(\tau^*)|\notag\\
&+|m_1(\tau)-m_1(\tau^*)||m_1(\tau)+m_1(\tau^*)|\le c(\mu)\eta(n;\tau).
\end{align}
Here $m_k(\tau^*),\,k=0,1,2$, denote moments of the measure $\tau^*$.

Let $X^*,X_1^*,X_2^*,\dots$ be free identically distributed random variables such that
$\Cal L(X^*)=\mu^*$. 
Denote $\mu_n^*:=\Cal L((X_1^*+\dots+X_n^*)/\sqrt {n})$. As before, 
by Proposition~\ref{3.3pro}, there exists $W(z)\in\mathcal F$
such that (\ref{3.10}) holds with $Z=W$ and $\mu=\mu^*$, and $F_{(\mu^*)^{n\boxplus}}(z)=F_{\mu^*}(W(z))$.
Hence $F_{\mu_n^*}(z)=F_{\mu^*}(\sqrt {n} T_n(z))/\sqrt {n},\,z\in\mathbb C^+$, 
where $T_n(z):=W(\sqrt {n} z)/\sqrt {n}$. In the sequel we shall need more detailed information about
the behaviour of the functions $T_n(z)$ and $S_n(z)$. By Lemma~\ref{l7.4}, 
these functions are continuous up to the real axis for $n\ge n_1(\mu)$. 
Their values for $z=x\in\mathbb R$ we denote by 
$T_n(x)$ and $S_n(x)$, respectively.
In order to formulate the following results for $T_n(z)$ 
we introduce some notations. Denote by $M_n(z)$ the reciprocal Cauchy transform of the free Meixner
measure $\mu_{a_n,b_n,d_n}$ with the parameters $a_n,b_n$ and $d_n$ from (\ref{2.3****}), i.e.,
$$
M_{n}(z):=a_n+\frac 12\Big(\big(1+b_n\big)(z-a_n)+
\sqrt{\big(1-b_n\big)^2(z-a_n)^2-4\big(1-d_n\big)}\Big),\quad z\in\mathbb C^+.
$$
Denote by $D_n$ the rectangle
\begin{equation}\notag
D_n=\Big\{z\in\mathbb C:0<\Im z\le 3,|\Re z-a_n|\le \frac 2{e_n}-\frac {\varepsilon_{n1}}n\Big\},
\end{equation}
where $\varepsilon_{n1}:=c_1(\mu)(\eta(n;\tau)
+1/\sqrt n)$ and $c_1(\mu)>0$ 
is sufficiently large. In the sequel we assume that $\varepsilon_{n1}$ is always of this form.

Repeating step by the step the arguments of Section~7 (see Subsections 7.2--7.7) of our paper~\cite{ChG:2013}
we establish the following result.
\begin{theorem}\label{th4a}
Let $\mu\in\mathcal M$ such that $m_4(\mu)<\infty$ and $m_1(\mu)=0,\,m_2(\mu)=1$.
Then
there exists a constant $c(\mu)$ such that the following relation holds, for $z\in D_n$ and $n\ge n_1(\mu)$,
\begin{align}
T_n(z)&=M_{n}(z)+\frac {\varepsilon_{n1}}n\,\frac {c(\mu)\,\theta}{\sqrt{(e_n(z-a_n))^2-4}}, \label{th4.1}\\
G_{\mu_n^*}(z)&=\frac 1{T_n(z)}+\frac {m_2(\mu^*)}{nT_n(z)^3}+\frac{c(\mu)\,\theta}{n^{3/2}}.\label{th4.2}
\end{align}
In addition
\begin{equation}\label{th4.1*}
0\le \Im T_n(x)\le c(\mu)\sqrt{\frac {\varepsilon_{n1}}n},\quad \frac 2{e_n}-\frac {\varepsilon_{n1}}n<|x-a_n|\le 3. 
\end{equation}
\end{theorem}

In (\ref{th4.1}) and (\ref{th4.2}) $\theta$ is a complex-valued quantities such that $|\theta|\le 1$.

Here and in the sequel constants $c(\mu)>0$ do not depend on the constant $c_1(\mu)$.

In Section~7 of this paper we shall give a more detailed exposition of the proof of this theorem.

By Lemmas~\ref{l7.4},~\ref{l7.5},
$|T_n(z)|\ge 1.03/3$ for $z\in \mathbb C^+\cup \mathbb R$ and for $n\ge n_1(\mu)$. It is obvious that
the same estimate holds for $M_n(z)$. Since 
\begin{equation}\label{th4.2^*}
G_{\mu_n^*}(z)=\sqrt {n} G_{\mu^*}(\sqrt {n} T_n(z))=\int\limits_{[-\sqrt{n-1}/3,\sqrt{n-1}/3]}
\frac{\mu^*(du)}{T_n(z)-u/\sqrt {n}},\qquad z\in\mathbb C^+, 
\end{equation}
we conclude that $G_{\mu_n^*}(z)$ is a continuous function up to the real axis. Denote its value for real $x$  
by $G_{\mu_n^*}(x)$.
Denote $G_{\hat{\mu}_n}(z):=1/T_n(z),\,z\in\mathbb C^+$. This function is continuous up to the real axis
as well. Therefore
$\hat{\mu}_n$ and $\mu_n^*$ are absolutely continuous measures with continuous densities
$\hat{p}_n(x)$ and $p_n^*(x)$, respectively,
\begin{align}
&\hat{p}_n(x)=-\lim_{\varepsilon\downarrow 0}\frac 1{\pi}\Im \frac 1{T_n(x+i\varepsilon)} 
=-\frac 1{\pi}\Im \frac 1{T_n(x)},\notag\\
&p_n^*(x)=-\lim_{\varepsilon\downarrow 0}\frac 1{\pi}\Im G_{\mu_n^*}(x+i\varepsilon)
=-\frac 1{\pi}\Im G_{\mu_n^*}(x).\notag
\end{align}

In addition, $\hat{p}_n(x)\le 1$ and $p_n^*(x)\le 50$ for all $x\in \mathbb R$ and $n\ge n_1(\mu)$.  
\begin{theorem}\label{th5}
Let $\mu\in\mathcal M$ such that $m_4(\mu)<\infty$ and $m_1(\mu)=0,\,m_2(\mu)=1$.
Then, for $x\in I_n:=\{x\in\mathbb R:|x-a_n|\le\frac 2{e_n}-\frac{\varepsilon_{n1}}n\}$ and $n\ge n_1(\mu)$, 
the following relation holds
\begin{align}\label{loc.6}  
p_n^*(x)=v_n(x-a_n)
+\frac{\varepsilon_{n1}}n\,\frac {c(\mu)\,\theta}{\sqrt{4-(e_n(x-a_n))^2}}, 
\end{align}
where $v_n(x)$ is defined in $(\ref{vden})$. 
\end{theorem}

\begin{proof}
We shall use the following estimate, for $x\in\mathbb R$,   
\begin{align}\label{loc.1}
|p_n^*(x)-p_{\mu_{a_n,b_n,d_n}}(x)-\frac 1n q_{n}(x)|&\le |\hat{p}_n(x)-p_{\mu_{a_n,b_n,d_n}}x)|
+|p_n^*(x)-\hat{p}_n(x)-\frac 1n \hat{q}_{n}(x)|\notag\\
&+\frac 1n|q_{n}(x)-\hat{q}_n(x)|, 
\end{align}
where 
\begin{equation}\notag
q_{n}(x):=-\frac 1{\pi}\Im \frac 1{M_n(x)^3},\quad\text{and}\quad
\hat{q}_{n}(x):=-\frac 1{\pi}\Im \frac 1{T_n(x)^3}.
\end{equation}
By (\ref{th4.1}) and by the lower bounds
\begin{equation}\label{loc.1*}
 |T_n(x)|\ge 1.03/3,\qquad |M_n(x)|\ge 1.03,\quad x\in\mathbb R,
\end{equation}
we easily obtain the upper bound, for $x\in I_n$
and $n\ge n_1(\mu)$,
\begin{equation}\label{loc.2}
|\hat{p}_n(x)-p_{\mu_{a_n,b_n,d_n}}x)|\le \frac{\varepsilon_{n1}}n\,\frac {c(\mu)}{\sqrt{4-(e_n(x-a_n))^2}} 
\end{equation}
and, by (\ref{th4.2}), we have
\begin{equation}\label{loc.3}
|p_n^*(x)-\hat{p}_n(x)-\frac 1n \hat{q}_{n}(x)|\le 
\frac{c(\mu)}{n^{3/2}},\quad x\in I_n,\quad n\ge n_1(\mu).
\end{equation}

Since, by (\ref{2.3g}),
\begin{equation}
p_{\mu_{a_n,b_n,d_n}}(x):=\frac {\sqrt{4(1-d_n)-(1-b_n)^2(x-a_n)^2}}{2\pi(b_nx^2+a_n(1-b_n)x+1-d_n)},
\quad x\in \tilde{I}_n:=[a_n-2/e_n,a_n+2/e_n], 
\end{equation}
we easily conclude that
\begin{equation}\label{loc.4}
p_{\mu_{a_n,b_n,d_n}}(x)=
\Big(1+\frac{d_n}2-a_n^2-a_n(x-a_n)-(b_n-a_n^2)(x-a_n)^2\Big)p_w(e_n(x-a_n))
+\frac{c(\mu)\theta}{n^{3/2}} 
\end{equation}
for $x\in \tilde{I}_n$. Using again (\ref{th4.1}) and (\ref{loc.1*}),
we obtain
\begin{align}\label{loc.4a}
|q_n(x)-\hat{q}_n(x)|&\le \frac 1{\pi}|T_n(x)-M_n(x)|\Big(\frac 1{|M_n(x)T_n(x)^3|}+
\frac 1{|M_n(x)^2T_n(x)^2|}+\frac 1{|M_n(x)^3T_n(x)|}\Big)
\notag\\
&\le \frac{\varepsilon_{n1}}n\,\frac { c(\mu)}{\sqrt{4-(e_n(x-a_n))^2}},\quad x\in I_n.   
\end{align}
On the other hand it is not difficult to show that
\begin{align}
q_{n}(x)&:=\frac 1{8\pi}
\sqrt{(4(1-d_{n})-(1-b_n)^2(x-a_n)^2)_+}\notag\\
&\times\frac{3((1+b_n)x+(1-b_n)a_n)^2+(1-b_n)^2(x-a_n)^2-
4(1-d_n)}{(b_nx^2+(1-b_n)a_n x+1-d_{n})^3},\quad x\in\mathbb R,\notag
\end{align}
which leads to the relation
\begin{equation}\label{loc.5}
q_{n}(x)=((x-a_n)^2-1)p_w(e_n(x-a_n))
+c(\mu)\theta (|a_n|+n^{-1})  
\end{equation}
for $x\in \tilde{I}_n$.

Applying (\ref{loc.2}), (\ref{loc.3}), (\ref{loc.4a}) and (\ref{loc.4}), (\ref{loc.5}) to (\ref{loc.1}) 
we arrive at the statement of the theorem.
\end{proof}

\section{Local asymptotic expansion}

First we prove 
the auxiliary result.
\begin{theorem}\label{th4c}
Let $\mu\in\mathcal M$ such that $m_4(\mu)<\infty$ and $m_1(\mu)=0,\,m_2(\mu)=1$.
Then the following relation holds
\begin{equation}\label{th4.4*}
p_{n}(x)=p_{n}^*(x)+\tilde{\rho}_{n1}(x)+\tilde{\rho}_{n2}(x),\quad x\in \mathbb R,\,\,
n\ge n_1(\mu),
\end{equation} 
where
\begin{equation}\notag
|\tilde{\rho}_{n1}(x)|\le c(\mu)\big(|\Im (S_n(x)-T_n(x))|+ \Im T_n(x)|S_n(x)-T_n(x)|+n^{-2}\big)
\end{equation}
and $\tilde{\rho}_{n2}(x)$ is a continuous function such that 
$$
0\le \tilde{\rho}_{n2}(x)\le c(\mu)\quad\text{ and}\quad \int_{\mathbb R}\tilde{\rho}_{n2}(x)\,dx=o\Big(\frac 1{n^2}\Big).
$$
\end{theorem}

\begin{proof}
Represent the density $p_{n}(x)$ of the measure $\mu_n$ in the form
\begin{equation}\label{th4.4}
p_{n}(x)=p_{n1}(x)+p_{n2}(x),\quad x\in\mathbb R, 
\end{equation}
where $p_{nj}(x)\ge 0,\,x\in\mathbb R,\,j=1,2$, and, for $z\in\mathbb C^+$,
\begin{align}
I_1(z)&:=\int_{|u|\le \sqrt{n-1}/3}\frac{\mu(du)}{S_n(z)-u/\sqrt n}=\int_{\mathbb R}\frac{p_{n1}(u)\,du}{z-u},\notag\\
I_2(z)&:=\int_{|u|>\sqrt{n-1}/3}\frac{\mu(du)}{S_n(z)-u/\sqrt n}=\int_{\mathbb R}\frac{p_{n2}(u)\,du}{z-u}.\notag
\end{align} 
Since $\lim_{y\to+\infty}iyI_2(iy)=\int_{|u|>\sqrt{n-1}/3}\mu(du)=\int_{\mathbb R}p_{n2}(u)\,du$, we note that
\begin{equation}\label{th4.5}
\int_{\mathbb R}p_{n2}(u)\,du=o(n^{-2}). 
\end{equation}
Since $S_n(x),\,x\in\mathbb R$, is a continuous function and $|S_n(x)|\ge 1.03/3$
for all $x\in\mathbb R$, 
we easily see that, for $x\in\mathbb R$,
\begin{align}\notag
p_{n1}(x)=-\frac 1{\pi}\Im \int_{|u|\le \sqrt{n-1}/3}\frac{\mu(du)}{S_n(x)-u/\sqrt n} 
\end{align}
and that $p_{n1}(x)$ is a continuous function on the real line. In view of (\ref{asden3}), 
$p_{n}(x)$ is a continuous function on the real line and
$p_n(x)\le c(\mu),\,x\in\mathbb R$, 
for $n\ge n_1(\mu)$. Therefore we conclude from (\ref{th4.4}) that $p_{n2}(x)$ is a continuous function 
on the real line and $p_{n2}(x)\le c(\mu),\,x\in\mathbb R$, for the same $n$.

Now we may write
\begin{align}
\pi(p_{n}^*(x)-p_{n1}(x))&=I_{3,1}(x)+I_{3,2}(x)\notag\\
&:=\Im\Big(\int_{|u|\le \sqrt{n-1}/3}\frac{\mu(du)}{S_n(x)-u/\sqrt n}-
\int_{|u|\le \sqrt{n-1}/3}\frac{\mu(du)}{T_n(x)-u/\sqrt n}\Big)\notag\\
&+\Im\int_{|u|\le \sqrt{n-1}/3}\frac{(\mu-\mu^*)(du)}{T_n(x)-u/\sqrt n}. \label{th4.6^*}
\end{align}

Since 
\begin{align}
\frac {\Im S_n(x)}{|S_n(x)-u/\sqrt n|^2}&-\frac {\Im T_n(x)}{|T_n(x)-u/\sqrt n|^2}
=\frac{\Im S_n(x)-\Im T_n(x)}{|S_n(x)-u/\sqrt n|^2}\notag\\
&+\Im T_n(x)\Big(\frac 1{|S_n(x)-u/\sqrt n|^2}
-\frac 1{|T_n(x)-u/\sqrt n|^2}\Big),\notag
\end{align}
we have
\begin{align}
&I_{3,1}(x)=(\Im S_n(x)-\Im T_n(x))\int\limits_{|u|\le \sqrt{n-1}/3}\frac{\mu(du)}{|S_n(x)-u/\sqrt n|^2}\notag\\
&+\Im T_n(x)\Re(T_n(x)-S_n(x))\int\limits_{|u|\le \sqrt{n-1}/3}\frac{(\Re S_n(x)+\Re T_n(x)-2u/\sqrt n)\,\mu(du)}
{|S_n(x)-u/\sqrt n|^2|T_n(x)-u/\sqrt n|^2}\notag\\
&+\Im T_n(x)\Im(T_n(x)-S_n(x))\int\limits_{|u|\le \sqrt{n-1}/3}\frac{(\Im S_n(x)+\Im T_n(x)))\,\mu(du)}
{|S_n(x)-u/\sqrt n|^2|T_n(x)-u/\sqrt n|^2}.\notag
\end{align}

By the inequalities $|S_n(x)-u/\sqrt n|\ge 0.01$, $|T_n(x)-u/\sqrt n|\ge 0.01$
for $x\in\mathbb R$ and $|u|\le\sqrt{n-1}/3$, we conclude that, for $x\in \mathbb R$, 
\begin{equation}\label{th4.6a}
|I_{3,1}(x)|\le c(\mu)|\Im S_n(x)-\Im T_n(x)|+c(\mu)\Im T_n(x)|S_n(x)-T_n(x)|.
\end{equation}

On the other hand we note that, for $x\in\mathbb R$,
\begin{align}
I_{3,2}(x)&=\Im\Big(\frac{m_2(\mu)-m_2(\mu^*)}{T_n^3(x)n}+\frac{m_3(\mu)-m_3(\mu^*)}{T_n^4(x)n^{3/2}}\notag\\
&-\sum_{j=0}^3\frac 1{T_n(x)^{j+1}n^{j/2}}\int_{|u|>\sqrt{n-1}/3}u^j\,\mu(du)
+\frac 1{T_n^4(x)n^{2}}\int\limits_{|u|\le \sqrt{n-1}/3}\frac{u^4(\mu-\mu^*)(du)}{T_n(x)-u/\sqrt n}\Big). \notag
\end{align}
Using (\ref{Pas5.1})--(\ref{Pas5.1b}) and the inequality
\begin{equation}\notag
\int_{|u|>\sqrt{n-1}/3}|u|^j\,\mu(du)\le c(\mu)n^{-(4-j)/2},\quad j=0,1,2,3,4, 
\end{equation}
we obtain the estimate
\begin{equation}\label{th4.7^*}
|I_{3,2}(x)|\le \frac{c(\mu)}{n^{2}},\quad x\in\mathbb R. 
\end{equation}
Applying (\ref{th4.6a}) and (\ref{th4.7^*}) to (\ref{th4.6^*}), we have, for $x\in \mathbb R$, 
\begin{align}\label{th4.8}
|p_{n1}(x)-p_{n}^*(x)|\le c(\mu)\big(|\Im S_n(x)-\Im T_n(x)|+\Im T_n(x)|S_n(x)-T_n(x)|+n^{-2}\big).
\end{align}
The representation (\ref{th4.4*}) follows immediately from (\ref{th4.4}) if to define $\tilde{\rho}_{n1}(x)=p_{n1}(x)-p^*_{n}(x)$ and
$\tilde{\rho}_{n2}(x)=p_{n2}(x)$ and from the bounds (\ref{th4.8}) and (\ref{th4.5}). 
\end{proof}

\section{Proof of Theorem~\ref{th4a}}

In this section we show how the arguments of Section 7 in \cite{ChG:2013}
lead to a~proof of Theorem~\ref{th4a}. 

Repeating the arguments of Subsection~7.2 we deduce that $T_n(z)$ satisfies the functional equation,
for $z\in\mathbb C^+$,
\begin{equation}\label{5.5}
T_n^5(z)-zT_n^4(z)+m_2(\mu^*)T_n^3(z)
+\frac{\zeta_{n2}(z)}{\sqrt n}T_n^2(z)
+\frac{\zeta_{n3}(z)}n T_n(z)-\frac{\zeta_{n4}(z)z}{n^2} =0,
\end{equation}
where 
$
\zeta_{n1}(z):=\int_{\mathbb R}\frac{u^5\,\mu^*(du)}{W(\sqrt nz)-u},
$
$\zeta_{n2}(z):=m_3(\mu^*)-z/\sqrt n$, $\zeta_{n3}(z)(z):=m_4(\mu^*)+\zeta_{n1}(z)-zm_3(\mu^*)/\sqrt n$ and
$\zeta_{n4}(z)(z):=m_4(\mu^*)+\zeta_{n1}(z)$.
As in Subsection~7.3 from \cite{ChG:2013} we obtain estimates for the functions $\zeta_{nj}(z),j=1,2,3,4$, in the domain
$D^*:=\{|\Re z|\le 4,\,0<\Im z\le 3\}$
\begin{equation}\label{5.5a}
|\zeta_{n1}(z)|\le c(\mu)\delta_n\le c(\mu)\eta(n;\tau),\quad \sum_{j=2}^4 |\zeta_{nj}(z)|\le c(\mu).
\end{equation}

For every fixed $z\in\mathbb C^+$
consider the~equation
\begin{equation}\label{5.4c}
Q(z,w):=w^5-zw^4+m_2(\mu^*)w^3+\frac{\zeta_{n2}(z)}{\sqrt n} w^2
+\frac{\zeta_{n3}(z)}n w-\frac{\zeta_{n4}(z)z}{n^2} =0.
\end{equation}
Denote the~roots of the equation (\ref{5.4c}) by $w_j=w_j(z),\,j=1,\dots,5$.

As in Subsection~7.4 \cite{ChG:2013} we can show that for every fixed $z\in D^*$ 
the~equation $Q(z,w)=0$ 
has three roots, say $w_j=w_j(z),\,j=1,2,3$, such that
\begin{equation}\label{5.6}
|w_j|<r':=c_2(\mu) n^{-1/2},
\quad j=1,2,3,
\end{equation}
and two roots, say $w_j,\,j=4,5$, such that $|w_j|\ge r'$ for $j=4,5$.

Represent $Q(z,w)$ in the~form
$$
Q(z,w)=(w^2+s_1w+s_2)(w^3+g_1w^2+g_2w+g_3),
$$
where $w^3+g_1w^2+g_2w+g_3=(w-w_1)(w-w_2)(w-w_3)$.
From this formula we derive the~relations
\begin{align}\label{5.7}
s_1+g_1=-z,\quad &s_2+s_1g_1+g_2=m_2(\mu^*),\quad s_2g_1+s_1g_2+g_3=
\frac{\zeta_{n2}(z)}{\sqrt n},
\notag \\
&s_2g_2+s_1g_3=\frac {\zeta_{n3}(z)}n,\quad s_2g_3=-\frac{\zeta_{n4}(z)z}{n^2}.
\end{align}
By Vieta's formulae and (\ref{5.6}), note that
\begin{equation}\label{5.8a}
|g_1|\le 3r',\quad |g_2|\le 3(r')^2,\quad
|g_3|\le (r')^3.
\end{equation}
Now we obtain from (\ref{5.7}) and (\ref{5.8a}) the~following bounds,
for $z\in D^*$,
\begin{equation}\label{5.8}
|s_1|\le 5+3r',\quad |m_2(\mu^*)-s_2|\le 3r'(4r'+5)\le 16r'\le\frac 12.
\end{equation}
Then we conclude from 
(\ref {Pas5.1}), (\ref{5.5a}), (\ref{5.7})--(\ref{5.8}) that, for the~same $z$,
\begin{align}\label{5.9}
\Big|g_2-\frac{\zeta_{n4}(z)}{n}\Big| &\le \Big|g_2-\frac{\zeta_{n3}(z)}{n}\Big|+\frac{|m_3(\mu^*)||z|}{n^{3/2}}
\le \frac{|s_1|}{|s_2|}|g_3|+\frac{|s_2-m_2(\mu^*)|}{|s_2|}\frac{|\zeta_{n3}(z)|}n+(r')^3\notag\\
&\le 11(r')^3+8(r')^3+(r')^3=20(r')^3\le c(\mu)n^{-3/2}.
\end{align}
Now repeating the arguments of Subsection~7.4 we deduce the inequality
\begin{equation}\label{5.10}
 |g_1-a_n-b_n z|\le c(\mu)\frac{\eta(n;\tau)}n, \quad z\in D^*.
\end{equation}
To find the~roots $w_4$ and $w_5$, we need to solve the~equation
$w^2+s_1w+s_2=0$. Using (\ref{5.7}), we have, 
for $j=4,5$,
\begin{align}\label{5.11}
w_j&=\frac 12\Big(-s_1+(-1)^j\sqrt{s_1^2-4s_2}\Big)\notag\\&=\frac 12\Big(
z+g_1+(-1)^j\sqrt{(z+g_1)^2-4(m_2(\mu^*)+(z+g_1)g_1-g_2)}\Big)\notag\\
&=\frac 12\Big(z+g_1+(-1)^j\sqrt{(z-g_1)^2-4m_2(\mu^*)-4(g_1^2-g_2)}\Big)
=\frac 12 r_{n1}(z)+a_n+\notag\\
&+\frac 12\Big(\big(1+b_n\big)(z-a_n)
+(-1)^j\sqrt{\big(1-b_n\big)^2(z-a_n)^2-4\big(1-d_n\big)
+r_{n2}(z)}\Big),
\end{align}
where 
\begin{align}
r_{n1}(z)&:=g_1-a_n-b_n(z-a_n),\notag\\
r_{n2}(z)&:=
-3r_{n1}^2(z)-2r_{n1}(z)(4a_n+(1+3b_n)(z-a_n))+4(1-m_2(\mu^*))\notag\\
&+4(g_2-m_4(\mu^*)/n)-4b_n(z-a_n)(2a_n+b_n(z-a_n)).\notag 
\end{align}

The quantities $r_{n1}(z)$ and $r_{n2}(z)$ admit the bound (see Subsection~7.4 and 7.5 from \cite{ChG:2013})
\begin{equation}\label{5.12}
|r_{n1}(z)|+|r_{n2}(z)|\le c(\mu)\frac{\eta(n;\tau)}n,\quad z\in D^*.
\end{equation}
We choose the~branch of the~analytic square root according to the~condition $\Im w_4(i)\ge 0$.

As in Subsection~7.6 from \cite{ChG:2013} we prove that $w_4(z)=T_n(z)$ for $z\in D_n$, where 
the constant $c(\mu)$ in (\ref{5.12}) does not depend on the constant $c_1(\mu)$. 
Since the constant $c_1(\mu)$ is sufficiently large, we have, by (\ref{5.12}),
\begin{equation}\label{5.13}
|r_{n2}(z)|/|((1-b_n)^2(z-a_n)^2-4(1-d_n)|\le 10^{-2},\quad z\in D_n. 
\end{equation}

For $z\in D_n$, using formula (\ref{5.11}) with $j=4$ for $T_n(z)$, we write
\begin{align}\label{5.14}
&M_n(z)-T_n(z)=-\frac 12 r_{n1}(z)\notag\\
&-\frac 12\frac{r_{n2}(z)}{\sqrt{(1-b_n)^2(z-a_n)^2-4(1-d_n)}
+\sqrt{(1-b_n)^2(z-a_n)^2-4(1-d_n)+r_{n2}(z)}}. 
\end{align}
Using (\ref{5.13}) and the power expansion for the function $(1+z)^{1/2},\, |z|<1$, we easily
rewrite (\ref{5.14}) in the form
\begin{align}\label{5.15}
M_n(z)-T_n(z)=\frac{r_{n3}(z)}{\sqrt{(1-b_n)^2(z-a_n)^2-4(1-d_n)}},
\quad z\in D_n, 
\end{align}
where $|r_{n3}(z)|\le c(\mu)\eta(n;\tau)/n$. The relation (\ref{th4.1}) immediately follows from (\ref{5.15}). 

Using (\ref{th4.2^*}), we conclude that
\begin{equation}\label{5.2}
G_{\mu_n^*}(z)=\frac 1{T_n(z)}+\frac {m_2(\mu^*)}{nT_n^3(z)}
+\frac {1}{n^{3/2}T_n^3(z)}
\int_{[-\sqrt{n-1}/3,\sqrt{n-1}/3]}\frac{u^3\,\mu^*(du)}{T_n(z)-u/\sqrt n},\quad z\in\mathbb C^+.
\end{equation}
Since $|T_n(z)|\ge 1.03/3$ for $z\in \mathbb C^+$ and $n\ge n_1(\mu)$, we arrive at (\ref{th4.2}).

The function $T_n(x)$ for real $x$ such that $\frac 2{e_n}-\frac{\varepsilon_{n1}}n\le |x-a_n|\le 3$ coincide
with $w_3(x)$ or $w_4(x)$ from (\ref{5.11}). Here we understand $w_j(x)$ as limit values of $w_j(z)$ where $z\in D^*$
and $z\to x$. It is not difficult to conclude from the formula (\ref{5.11}) that $0\le \Im T_n(x)\le c(\mu)\sqrt{\frac{\varepsilon_{n1}}n}$
for $\frac 2{e_n}-\frac{\varepsilon_{n1}}n\le |x-a_n|\le 3$.

\section{Proof of Theorem~\ref{th7}}

Recalling the definition of the function $Z(z)$, we see, by Lemma~\ref{l7.4}, that the function $S_n(z)$ 
maps $\mathbb C^+$ conformally onto $\hat{D}_n$, where 
$\hat{D}_n:=\{z=x+iy,x,y\in\mathbb R:y>y_n(x\sqrt n)/\sqrt n\}$. Denote by $\hat{\gamma}_n$ a curve 
given by the equation $z=x+i\hat{y}_n(x)$, where $\hat{y}_n(x)=y_n(x\sqrt n)/\sqrt n$. The function $S_n(z)$ 
is continuous up to the real axis and it establishes a~homeomorphism between 
the real axis and the curve $\hat{\gamma}_n$.

Note as well that the function $T_n(z)$ maps $\mathbb C^+$ conformally onto $\hat{D}_n^*$,  
where $\hat{D}_n^*:=\{z=x+iy,x,y\in\mathbb R:y>y_n^*(x\sqrt n)/\sqrt n\}$. Here $y_n^*(x)$ 
is defined in the same way as $y_n(x)$ if we change the measure $\mu$ by $\mu^*$. 
By definition of the measure $\mu^*$, we see that $y_n^*(x)\le y_n(x)$. Hence 
$S_n(\mathbb C^+)\subseteq T_n(\mathbb C^+)$. Denote by $\hat{\gamma}_n^*$ a curve 
given by the equation $z=x+i\hat{y}_n^*(x),\,x\in\mathbb R$, where $\hat{y}_n^*(x)=y_n^*(x\sqrt n)/\sqrt n$. 
The function $T_n(z)$ 
is continuous up to the real axis and it establishes a~homeomorphism between 
the real axis and the~curve~$\hat{\gamma}_n^*$. 

Let $x\in\mathbb R$. Since $S_n(z)$ and $T_n(z)$ are the conformal maps $\mathbb C^{+}$ on $\hat{D}_n$ and
$\hat{D}_n^*$, respectively, which are continuous up to the real axis, we note that the functions $\Re S_n(x)$ and $\Re T_n(x)$
are monotonically increasing. Hence for every $x\in\mathbb R$ there exists unique $\tilde{x}\in\mathbb R$ such that
$\Re S_n(x)=\Re T_n(\tilde{x})$. Denote $h(x):=\Im S_n(x)-\Im T_n(\tilde{x})\ge 0$.

In order to prove Theorem~\ref{th7} we need the following auxiliary results. In the sequel
we assume that $n\ge n_1(\mu)\ge 10$.
\begin{proposition}\label{th4bpro1}
For $z_1,z_2\in T_n(\mathbb C^+)$,
\begin{equation}\notag
|T_n^{(-1)}(z_1)-T_n^{(-1)}(z_2)|\le c(\mu)|z_1-z_2|.
\end{equation} 
\end{proposition}
\begin{proof}
Using the formula
\begin{equation}\label{th4bpro1.1}
 T_n^{(-1)}(z)=nz-\frac{n-1}{\sqrt n}F_{\mu^*}(z\sqrt n)=z-\frac{n-1}{\sqrt n}\int_{-\delta_n\sqrt{n-1}}^{\delta_n\sqrt{n-1}}
\frac{\tau(du)}{u-z\sqrt n},\quad z\in\mathbb C^+,
\end{equation}
we have the relation, for $z_1,z_2\in\mathbb C^+$,
\begin{equation}\notag
T_n^{(-1)}(z_1)-T_n^{(-1)}(z_2)=(z_1-z_2)\Big(1-(n-1)\int_{-\delta_n\sqrt{n-1}}^{\delta_n\sqrt{n-1}}
\frac{\tau(du)}{(u-z_1\sqrt n)(u-z_2\sqrt n)}\Big). 
\end{equation}
Since, by Lemma~3.5, $|u-z\sqrt n|\ge 10^{-2}\sqrt{n}$ for $|u|\le \sqrt{n-1}/\pi$ and $z\in T_n(\mathbb C^+)$, we 
immediately arrive at the assertion of the proposition.
\end{proof}

\begin{proposition}\label{th4bpro1a}
For $z\in S_n(\mathbb C^+)$,
\begin{equation}\notag
|T_n^{(-1)}(z)-S_n^{(-1)}(z)|\le \frac{\tau(\{|u|>\delta_n\sqrt{n-1}\})}{\Im z}.
\end{equation} 
\end{proposition}
\begin{proof}
Using (\ref{th4bpro1.1}) and the formula
\begin{equation}\label{th4bpro3.1}
S_n^{(-1)}(z)=nz-\frac{n-1}{\sqrt n}F_{\mu}(z\sqrt n)=z-\frac{n-1}{\sqrt n}\int_{\mathbb R}
\frac{\tau(du)}{u-z\sqrt n},\quad z\in\mathbb C^+,
\end{equation} 
we have, taking into account that $S_n(\mathbb C^+)\subseteq T_n(\mathbb C^+)$,
\begin{equation}\notag
|S_n^{(-1)}(z)-T_n^{(-1)}(z)|=\frac{n-1}{\sqrt n}\Big|\int_{|u|>\delta_n\sqrt{n-1}}\frac{\tau(du)}{u-z\sqrt n}\Big| \le
\frac{\tau(\{|u|>\delta_n\sqrt{n-1}\})}{\Im z} 
\end{equation}
for $z\in S_n(\mathbb C^+)$, proving the proposition.
\end{proof}
\begin{proposition}\label{th4bpro1b}
For $x_1,x_2\in I_n$, we have the estimate 
\begin{equation}\notag
|T_n(x_1)-T_n(x_2)|\le c(\mu)\frac{|x_1-x_2|+\varepsilon_{n1}/n}
{\min_{j=1,2}\{\sqrt{4-(e_n(x_j-a_n))^2}\}}.
\end{equation} 
\end{proposition}
\begin{proof}
By Theorem~\ref{th4a}, we have the following relation
\begin{align}\label{th4bpro1b.1}
T_n(x_1)-T_n(x_2)&=M_n(x_1)-M_n(x_2)+\frac{\varepsilon_{n1}}n\,\frac{c(\mu)\,\theta}{\sqrt{(e_n(x_1-a_n))^2-4}}\notag\\
&-\frac{\varepsilon_{n1}}n\,\frac{c(\mu)\,\theta}{\sqrt{(e_n(x_2-a_n))^2-4}},
\end{align}
where $x_1,x_2\in I_n$.
On the other hand it is easy to see that
\begin{align}
&M_n(x_1)-M_n(x_2)=\frac{(1+b_n)(x_1-x_2)}2\notag\\ 
&+\frac 12\frac{(1-b_n)^2(x_1-x_2)(x_1+x_2-2a_n)}
{\sqrt{(1-b_n)^2(x_1-a_n)^2-4(1-d_n)}+\sqrt{(1-b_n)^2(x_2-a_n)^2-4(1-d_n)}}.\label{th4bpro1b.2}
\end{align} 
Moreover, we have, for $x_1,x_2\in I_n$,
\begin{align}
&\Big|\sqrt{(1-b_n)^2(x_1-a_n)^2-4(1-d_n)}
+\sqrt{(1-b_n)^2(x_2-a_n)^2-4(1-d_n)}\Big|\notag\\
&=(|\sqrt{4(1-d_n)-(1-b_n)^2(x_1-a_n)^2}|+|\sqrt{4(1-d_n)-(1-b_n)^2(x_2-a_n)^2}|). \notag
\end{align} 
In view of this relation and (\ref{th4bpro1b.1}), (\ref{th4bpro1b.2}), we easily obtain the assertion of the proposition.
\end{proof}

\begin{proposition}\label{th4bpro1c}
For $x\in I_n$, 
the following formula holds
\begin{align}
T_n(x)=a_n&+\frac 12\Big((1+b_n)(x-a_n)+(1-b_n)\sqrt{(x-a_n)^2-4/e_n^2}\Big)\notag\\
&+\frac{\varepsilon_{n1}}n\,\frac{c(\mu)\,\theta}{\sqrt{(e_n(x-a_n))^2-4}}.  \notag
\end{align}
\end{proposition}
\begin{proof}
The proof immediately follows from Theorem~\ref{th4a}. 
\end{proof}

\begin{proposition}\label{th4bpro3a}
For $x\in\mathbb R$, the following estimates hold
\begin{align}
 |S_n(x)|&\le |x|+\sqrt{\frac {n}{n-1}},\label{th4bpro3a1}\\
 |x-\tilde{x}|&\le 2\sqrt{\frac {n}{n-1}}.\label{th4bpro3a2}
\end{align}
\end{proposition}
\begin{proof}
We note from (\ref{3.10}) that 
$
S_n(z)=\frac zn+\frac{n-1}nF_{\mu_n}(z),
$
where $\mu_n$ is the distribution of $Y_n$. By Proposition~\ref{l7.6}, we see that $S_n(z)=z-G_{\nu_{n-1}\boxplus w_t}(z),\,z\in\mathbb C^+\cup\mathbb R$,
where the measure $\nu_{n-1}$ is given by $d\nu_{n-1}(x)=d\nu(\sqrt n x)$, where, by (\ref{Pas5.0a}), $\nu=\tau$,
and $t=t(n)=(n-1)/n$. Since, by Proposition~\ref{l7.7}, $|G_{\nu_{n-1}\boxplus w_t}(z)|
\le \sqrt{\frac n{n-1}}, \,z\in\mathbb C^+\cup\mathbb R$, we obtain the upper bound (\ref{th4bpro3a1}). 

Now we note that in the same way as above $T_n(z)=z-G_{\zeta_{n-1}\boxplus w_t}(z),\,z\in\mathbb C^+\cup\mathbb R$,
where the measure $\zeta_{n-1}$ is given by $d\zeta_{n-1}(x)=d\zeta(\sqrt n x)$, where, by (\ref{Pas5.0b}), $\zeta$
a narrowing of the measure $\tau$ on the interval $[-\delta\sqrt{n-1},\delta\sqrt{n-1}]$. Moreover $|G_{\zeta_{n-1}\boxplus w_t}(z)|
\le \sqrt{\frac n{n-1}}, \,z\in\mathbb C^+\cup\mathbb R$.

It remains to write the following relation
\begin{equation}\notag
\tilde{x}-\Re G_{\zeta_{n-1}\boxplus w_t}(\tilde{x})=\Re T_n(\tilde{x})=\Re S_n(x)=x-G_{\nu_{n-1}\boxplus w_t}(x)
\end{equation}
and obtain from here the upper bound (\ref{th4bpro3a2}). The proposition is proved.
\end{proof}
 
\begin{proposition}\label{th4bpro3}
Let $x\in \mathbb R$. 
Then
\begin{equation}\label{th4b.1*}
h(x)\le c(\mu)\tau(\{|u|>\delta_n\sqrt{n-1}\})\frac{1+|S_n(x)|^2}{(\Im S_n(x))^3}. 
\end{equation}
\end{proposition}
\begin{proof} 
Without loss of generality we assume that $\Im T_n(\tilde{x})>0$. The case $\Im T_n(\tilde{x})=0$ considers in the same way.
It follows from (\ref{th4bpro1.1}) that
\begin{equation}\label{th4bpro1.4*}
1=\frac{n-1}{n}\int_{-\delta_n\sqrt{n-1}}^{\delta_n\sqrt{n-1}}
\frac{\tau(du)}{(\Re T_n(\tilde{x})-u/\sqrt n)^2+(\Im T_n(\tilde x))^2}.
\end{equation}
The formula (\ref{th4bpro3.1}) gives us
\begin{equation}\label{th4bpro1.5*}
1=\frac{n-1}{n}\Big(\int_{-\delta_n\sqrt{n-1}}^{\delta_n\sqrt{n-1}}+\int_{|u|>\delta_n\sqrt{n-1}}\Big)
\frac{\tau(du)}{(\Re T_n(\tilde{x})-u/\sqrt n)^2+(\Im T_n(\tilde x)+h(x))^2}.
\end{equation}

For $x\in\mathbb R$  
we have the following 
lower bound
\begin{align}\label{th4bpro3.2}
\tilde{I}_1&:=\int_{|u|\le \delta_n\sqrt{n-1}}\frac{\tau(du)}{(u/\sqrt n-\Re T_n(\tilde{x}))^2+(\Im T_n(\tilde x))^2}\notag\\
&-\int_{|u|\le \delta_n\sqrt{n-1}}\frac{\tau(du)}{(u/\sqrt n-\Re T_n(\tilde{x}))^2+(\Im T_n(\tilde x)+h(x))^2}\notag\\
&\ge \int_{|u|\le \delta_n\sqrt{n-1}}\frac{nh(x)(2\Im T_n(\tilde x)+h(x)) \,\tau(du)}{(u/\sqrt n-\Re S_n(x))^2+(\Im S_n(x))^2}
\ge\frac {c(\mu)h(x)\Im S_n(x)}{1+|S_n(x)|^2}
\end{align}
and the upper bound
\begin{equation}\label{th4bpro3.3}
\tilde{I}_2:=\int_{|u|>\delta_n\sqrt{n-1}}\frac{\tau(du)}{(\Re T_n(\tilde{x})-u/\sqrt n)^2+(\Im T_n(\tilde x)+h(x))^2}\le
\frac{\tau(\{|u|>\delta_n\sqrt{n-1}\})}{(\Im S_n(x))^2}.
\end{equation}
It follows from (\ref{th4bpro1.4*}) and (\ref{th4bpro1.5*}) that $\tilde{I}_1=\tilde{I}_2$, therefore we obtain the assertion of the proposition from
(\ref{th4bpro3.2}) and (\ref{th4bpro3.3}).
\end{proof}

\begin{proposition}\label{th4bpro4*}
Let $x\in \mathbb R$. 
Then
\begin{equation}\label{th4b.2*}
|x-\tilde {x}|\le c(\mu)\tau(\{|u|>\delta_n\sqrt{n-1}\})\frac{1+|S_n(x)|^2}{(\Im S_n(x))^{2}}.  
\end{equation}
\end{proposition}
\begin{proof}
Without loss of generality we assume that $\Im S_n(x)>0$.
By the formula (\ref{th4bpro1.1}), we have
\begin{equation}\label{th4bpro1.1*}
\tilde{x}=\Re T_n^{(-1)}(T_n(\tilde{x}))=\Re T_n(\tilde{x})+\frac{n-1}{n}\int_{-\delta_n\sqrt{n-1}}^{\delta_n\sqrt{n-1}}
\frac{(\Re T_n(\tilde{x})-u/\sqrt n)\,\tau(du)}{(\Re T_n(\tilde{x})-u/\sqrt n)^2+(\Im T_n(\tilde x))^2}.
\end{equation}
On the other hand, by (\ref{th4bpro3.1}), we obtain
\begin{equation}\label{th4bpro1.2*}
x=\Re S_n^{(-1)}(T_n(\tilde{x})+ih(x))=\Re T_n(\tilde{x})+\frac{n-1}{n}\int_{\mathbb R}
\frac{(\Re T_n(\tilde{x})-u/\sqrt n)\,\tau(du)}{(\Re T_n(\tilde{x})-u/\sqrt n)^2+(\Im S_n(x))^2}.
\end{equation}
It follows from these formulae that
\begin{equation}\label{th4bpro1.3*}
x=\tilde{x}+\frac{n-1}{n}(J_1(x)+J_2(x)), 
\end{equation}
where
\begin{equation}\notag
J_1(x):=-\int_{-\delta_n\sqrt{n-1}}^{\delta_n\sqrt{n-1}}
\frac{(\Re T_n(\tilde{x})-u/\sqrt n)h(x)(\Im S_n(x)+\Im T_n(\tilde {x}))\,\tau(du)}
{((\Re T_n(\tilde{x})-u/\sqrt n)^2+(\Im T_n(\tilde x))^2)
((\Re T_n(\tilde{x})-u/\sqrt n)^2+(\Im S_n(x))^2}
\end{equation}
and
\begin{equation}\notag
J_2(x):=\int_{|u|>\delta_n\sqrt{n-1}}\frac{(\Re T_n(\tilde{x})-u/\sqrt n)\,\tau(du)}
{(\Re T_n(\tilde{x})-u/\sqrt n)^2+(\Im S_n(x))^2}. 
\end{equation}
It is easy to see that
\begin{equation}\notag
|J_1(x)|\le c(\mu)h(x)\Im S_n(x)\quad\text{and}\quad
|J_2(x)|\le c(\mu)\frac{\tau(\{|u|>\delta_n\sqrt{n-1}\})}{\Im S_n(x)}. 
\end{equation}
Therefore, using (\ref{th4b.1*}), we have
\begin{align}
|x-\tilde{x}|&\le c(\mu)\tau(\{|u|>\delta_n\sqrt{n-1}\})\Big(\frac{1+|S_n(x)|^2}{(\Im S_n(x))^{2}}+\frac 1{\Im S_n(x)}\Big)\notag\\
&\le c(\mu)\tau(\{|u|>\delta_n\sqrt{n-1}\})\frac{1+|S_n(x)|^2}{(\Im S_n(x))^{2}}
\end{align}
and (\ref{th4b.2*}) is proved.
\end{proof}
\begin{proposition}\label{th4bpro4**}
For $x\in I_n^*$, the following inequalities hold 
\begin{equation}
\frac 12((2/e_n)^2-(x-a_n)^2)\le (2/e_n)^2-(\tilde{x}-a_n)^2\le \frac 32((2/e_n)^2-(x-a_n)^2). \notag
\end{equation}
\end{proposition}
\begin{proof}
Consider $x$ such that $|\tilde{x}-a_n|\le \frac 2{e_n}-\sqrt{\frac{\varepsilon_{n1}}{c_1(\mu)n}}$.
By (\ref{th4bpro3a1}) and (\ref{th4bpro3a2}) we have $|x-\tilde{x}|\le 3$ and $|S_n(x)|\le 7$.
In view of Proposition~\ref{th4bpro1c} and (\ref{Pas5.1}), it is easy to see that
\begin{equation}\notag
 \Im T_n(\tilde{x})\ge \frac 14\Big(\frac{\varepsilon_{n1}}{c_1(\mu) n}\Big)^{1/4}\ge c(\mu)(\tau(\{|u|>\delta_n\sqrt{n-1}\}))^{1/4}.
\end{equation}
By (\ref{th4b.2*}), we see that
\begin{equation}\notag
 |x(\tilde x)-\tilde x|\le c(\mu)\tau(\{|u|>\delta_n\sqrt{n-1}\})\frac{1+|S_n(x)|^2}{(\Im T_n(\tilde x))^{2}}
 \le c(\mu)\sqrt{\tau(\{|u|>\delta_n\sqrt{n-1}\})}. 
\end{equation}
Since $c(\mu)$ does not depend on $c_1(\mu)$, we conclude finally
\begin{equation}\label{th4bpro4**.1}
|x(\tilde x)-\tilde x|\le\frac 1{100}\sqrt{\frac{\varepsilon_{n1}}n}. 
\end{equation}
The function $x(\tilde x)$ is monotone, continuous and the assertion of the proposition follows
at once from (\ref{th4bpro4**.1}).
\end{proof}
\begin{proposition}\label{th4bpro4***}
We have the bounds
\begin{equation}\label{th4bpro4***.1}
|\Im S_n(x)-\Im T_n(x)|\le \frac{\varepsilon_{n1}}n\,\frac {c(\mu)}{((2/e_n)^2-(x-a_n)^2)^{3/2}},\quad x\in I_{n}^*.
\end{equation}
\end{proposition}
\begin{proof}
Let $x\in I_{n}^*$. 
We have the formula
\begin{align}\label{th4bpro4***.2}
\Im S_n(x)-\Im T_n(x)=\Im S_n(x)-\Im T_n(\tilde x)+\Im T_n(\tilde x)-\Im T_n(x)=h(x)+\Im T_n(\tilde x)-\Im T_n(x).
\end{align}

By Propostion~\ref{th4bpro4**}, if $x\in I_n^*$, then $\tilde{x}\in I_n$. 
We see, by Propostion~\ref{th4bpro1c}, that, for such $x$,
\begin{align}
\Im T_n(x)&=\frac 12\sqrt{(1-b_n)^2(x-a_n)^2-4(1-d_n)}+\frac{\varepsilon_{n1}}n\,\frac{c(\mu)\,\theta}{\sqrt{(e_n(x-a_n))^2-4}},\notag\\
\Im T_n(\tilde{x})&=\frac 12\sqrt{(1-b_n)^2(\tilde{x}-a_n)^2-4(1-d_n)}+\frac{\varepsilon_{n1}}n\,\frac{c(\mu)\,\theta}{\sqrt{(e_n(\tilde{x}-a_n))^2-4}}.\label{th4bpro4***.3}
\end{align}
Using again Proposition~\ref{th4bpro4**} we obtain
\begin{align}
(\Im T_n(\tilde x))^2-(\Im T_n(x))^2&=\frac 14(1-b_n)^2(x-\tilde x)(x+\tilde x-2a_n)
+\frac{\varepsilon_{n1}}n\,\sqrt{1-d_n}\notag\\
&+\Big(\frac{\varepsilon_{n1}}n\Big)^2\,\frac{c(\mu)\,\theta}{(e_n(x-a_n))^2-4)}.\notag
\end{align}

By Propositions~\ref{th4bpro3}-- \ref{th4bpro4**}, and the formula (\ref{th4bpro4***.3}), we conclude that
\begin{align}
h(x)&\le c(\mu)\tau(\{|u|>\delta_n\sqrt{n-1}\})/(\Im T_n(\tilde x))^3 \le c(\mu)\tau(\{|u|>\delta_n\sqrt{n-1}\})/(\Im T_n(x))^3 \notag\\
&\le c(\mu)\tau(\{|u|>\delta_n\sqrt{n-1}\})/((2/e_n)^2-(x-a_n)^2)^{3/2}, \label{th4bpro4***.4}\\
|x-\tilde x|&\le c(\mu)\tau(\{|u|>\delta_n\sqrt{n-1}/3\})/((2/e_n)^2-(x-a_n)^2).\label{th4bpro4***.5}
\end{align}
Therefore, using (\ref{th4bpro4***.5}), we get
\begin{align}\label{th4bpro4***.6}
|\Im T_n(\tilde x)-\Im T_n(x)|&=\frac{|(\Im T_n(\tilde x))^2-(\Im T_n(x))^2|}{\Im T_n(\tilde x)+\Im T_n(x)}\le 
\frac{|(\Im T_n(\tilde x))^2-(\Im T_n(x))^2|}{\Im T_n(x)}\notag\\
&\le \frac{c(\mu)|\tilde{x}-x|}{\Im T_n(x)}+\frac{\varepsilon_{n1}}n\,\frac{c(\mu)}{\Im T_n(x)}\notag\\
&\le\frac{c(\mu)\tau(\{|u|>\delta_n\sqrt{n-1}\})}{(\Im T_n(x))^3}
+\frac{\varepsilon_{n1}}n\,\frac{c(\mu)}{\Im T_n(x)}
\le\frac{\varepsilon_{n1}}n\,\frac{c(\mu)}{(\Im T_n(x))^3}.
\end{align}
Applying (\ref{th4bpro4***.4}) and (\ref{th4bpro4***.6}) to (\ref{th4bpro4***.2}) we arrive at the assertion of the proposition.
\end{proof}

\begin{proposition}\label{th4bpro4}
We have the bounds
\begin{equation}\label{th4bpro4.1}
|\Re S_n(x)-\Re T_n(x)|\le \frac{\varepsilon_{n1}}n\,\frac {c(\mu)}{(2/e_n)^2-(x-a_n)^2},\quad x\in I_{n}^*.
\end{equation} 
\end{proposition}
\begin{proof}
Since $\Re S_n(x)-\Re T_n(x)=\Re T_n(\tilde{x})-\Re T_n(x)$, we conclude, using Proposition~\ref{th4bpro1c} 
and (\ref{th4bpro4***.5}),
\begin{align}
|\Re T_n(\tilde{x})-\Re T_n(x)|&\le c(\mu)|\tilde{x}-x|+ \frac{\varepsilon_{n1}}n\,\frac{c(\mu)}{(e_n(\tilde{x}-a_n))^2-4}
+ \frac{\varepsilon_{n1}}n\, \frac{c(\mu)}{(e_n(x-a_n))^2-4}\notag\\
&\le \frac{\varepsilon_{n1}}n\,\frac{c(\mu)}{(e_n(x-a_n))^2-4}
\end{align}
for $x\in I_{n}^*$. The proposition is proved.
\end{proof}

\begin{proposition}\label{th4bpro5}
For $x\in I_n\setminus I_{n}^*$ and $\alpha\in(0,1]$, we have
\begin{align}\label{th4bpro5.1}
|\Im S_n(x)-\Im T_n(x)|&\le  \frac{\varepsilon_{n1}}n\,\frac {c(\mu)}{((2/e_n)^2-(x-a_n)^2)^{3\alpha/2}}
+\sqrt{\frac{\varepsilon_{n1}}n}\,\frac{c(\mu)}{((2/e_n)^2-(x-a_n)^2)^{1/2}}\notag\\
&+c(\mu)\, ((2/e_n)^2-(x-a_n)^2)^{\alpha/2}.\notag
\end{align} 
\end{proposition}
\begin{proof}
Let $x,\tilde{x}\in I_n\setminus I_n^*$ and $\alpha\in(0,1]$. We have the two possibilities
\begin{equation}
 a) \,\,\Im S_n(x)\ge(\Im T_n(x))^{\alpha}\quad\text{or}\quad b)\,\,\Im S_n(x)<(\Im T_n(x))^{\alpha}.\notag
\end{equation}
Consider the case $a)$. Then, by (\ref{th4b.1*}), we have
\begin{equation}\label{th4bpro5.2a}
h(x)\le c(\mu)\frac{\tau(\{|u|>\delta_n\sqrt{n-1}\})}{(\Im T_n(x))^{3\alpha}}\le c(\mu)\,\frac{\tau(\{|u|>\delta_n\sqrt{n-1}\})}{((2/e_n)^2-(x-a_n)^2)^{3\alpha/2}}.
\end{equation}
In addition, repeating the argument of the proof of Proposition~\ref{th4bpro4***} and using the inequality 
$|\tilde{x}-x|\le \sqrt{\frac{\varepsilon_{n1}}n}$, we obtain
\begin{equation}\notag
|\Im T_n(\tilde{x})-\Im T_n(x)|\le c(\mu)\frac{|\tilde{x}-x|}{\Im T_n(x)}+\frac{\varepsilon_{n1}} {n}\,
\frac{c(\mu)}{\Im T_n(x)}\le\sqrt{\frac{\varepsilon_{n1}}n}\frac{c(\mu)}
{((2/e_n)^2-(x-a_n)^2)^{1/2}}.
\end{equation}
Hence in the case $a)$
\begin{equation}\label{th4bpro5.2}
|\Im S_n(x)-\Im T_n(x)|\le \frac{\varepsilon_{n1}}{n}\,\frac{c(\mu)}{((2/e_n)^2-(x-a_n)^2)^{3\alpha/2}}
+\sqrt{\frac{\varepsilon_{n1}}n}\,\frac{c(\mu)}{((2/e_n)^2-(x-a_n)^2)^{1/2}}.
\end{equation}

Now consider the case $b)$. In this case we have the following simple estimate
\begin{equation}\label{th4bpro5.3}
|\Im S_n(x)-\Im T_n(x)|\le (\Im T_n(x))^{\alpha}+\Im T_n(x))\le c(\mu)((2/e_n)^2-(x-a_n)^2)^{\alpha/2}.
\end{equation}

It remains to consider the case when $x\in I_n\setminus I_n^*$ and $\tilde{x}\not\in I_n\setminus I_n^*$.

Let $x_{n2}^*=a_n+2/e_n-\sqrt{\frac{\varepsilon_{n1}}n}$ and $x_{n2}=a_n+2/e_n-\frac{\varepsilon_{n1}}n$. 
Assume that $x_{n2}^*<x\le x_{n2}$ and $\tilde{x}\in I_n^*$. Assume as well that $a)$ holds. By (\ref{th4b.2*}), 
we see that
$|\tilde{x}(x_{n2}^*)-x_{n2}^*|\le c(\mu)\sqrt{\frac{\varepsilon_{n1}}n}$. In addition $\tilde{x}(x)$ is a monotone increasing function, therefore
$|\tilde{x}(x)-x|\le c(\mu)\sqrt{\frac{\varepsilon_{n1}}n}$. Repeating the previous estimates we obtain the bounds (\ref{th4bpro5.2}) and (\ref{th4bpro5.3})
in the considered case. We prove the bounds (\ref{th4bpro5.2}) and (\ref{th4bpro5.3}) for $x_{n1}<x\le x_{n1}^*$, where 
$x_{n1}^*=a_n-2/e_n+\sqrt{\frac{\varepsilon_{n1}}n},\,x_{n1}=a_n-2/e_n+\frac{\varepsilon_{n1}}n$ and $\tilde{x}\in I_n^*$ 
in the same way.

Without loss of generality let us assume now that $\tilde{x}>x_{n2}$. 
Then, by (\ref{th4b.2*}) and (\ref{th4bpro4***.3}), we have
\begin{equation}\notag
|x-\tilde {x}|\le c(\mu)\frac{\tau(\{|u|>\delta_n\sqrt{n-1}\})}{(\Im T_n(x))^{2\alpha}}\le c(\mu)\frac{\tau(\{|u|>\delta_n\sqrt{n-1}\})}
{((2/e_n)^2-(x-a_n)^2)^{\alpha/2}}\le 1.  
\end{equation}
Then, by (\ref{th4.1*}), we conclude $\Im T_n(\tilde{x})\le c(\mu)\sqrt{\frac{\varepsilon_{n1}}n}$ and hence
\begin{equation}\notag
|\Im T_n(x)-\Im T_n(\tilde{x})|\le T_n(x)+c(\mu)\sqrt{\frac{\varepsilon_{n1}}n}\le c(\mu)((2/e_n)^2-(x-a_n)^2)^{1/2}
+c(\mu)\sqrt{\frac{\varepsilon_{n1}}n}.
\end{equation}
Since in our case (\ref{th4bpro5.2a}) holds, we arrive at the upper bound
\begin{align}
|\Im S_n(x)-\Im T_n(x)|&\le\frac{\varepsilon_{n1}}n\,\frac{c(\mu)}{((2/e_n)^2-(x-a_n)^2)^{3\alpha/2}}\notag\\
&+c(\mu)((2/e_n)^2-(x-a_n)^2)^{1/2}+c(\mu)\sqrt{\frac{\varepsilon_{n1}}n}.\notag
\end{align}
In the case $b)$ we have obviously the estimate (\ref{th4bpro5.3}). The proposition is proved.
\end{proof}

\begin{proposition}\label{th4bpro6}
For $x\in I_n\setminus I_n^*$,
\begin{equation}\notag
 |\Im S_n(x)-\Im T_n(x)|\le \sqrt{\frac{\varepsilon_{n1}}n}\frac{c(\mu)}{\sqrt{(2/e_n)^2-(x-a_n)^2}}.
\end{equation}
\end{proposition}
\begin{proof}
Note that, for every fixed $x\in I_n\setminus I_n^*$,
\begin{equation}\notag
\frac{\varepsilon_{n1}}n\,\frac {1}{(2/e_n-|x-a_n|)^{3\alpha/2}}= (2/e_n-|x-a_n|)^{\alpha/2}
\end{equation}
for
\begin{equation}\notag
 \alpha=\frac 12\,\frac{\log(\varepsilon_{n1}/n)}{\log(2/e_n-|x-a_n|)}\le 1.
\end{equation}
Moreover, for this $\alpha$,
\begin{equation}\notag
(2/e_n-|x-a_n|)^{\alpha/2}=\Big(\frac{\varepsilon_{n1}}n\Big)^{1/4}. 
\end{equation}
Therefore the assertion of the proposition follows immediately from Proposition~\ref{th4bpro5}.
\end{proof}

Now we finish the proof of Theorem~\ref{th7}. Return to the formulations of Theorem~\ref{th5} and \ref{th4.4*}.
Denote
$$
\rho_{n1}(x):=\tilde{\rho}_{n1}(x)+\frac{\varepsilon_{n1}}n\,\frac{c(\mu)\,\theta}{((2/e_n)^2-(x-a_n)^2)^{1/2}}
\quad\text{and}\quad
\rho_{n2}(x)=\tilde{\rho}_{n2}(x)
$$
for $x\in I_n$. The statement of Theorem~\ref{th7} for $x\in I_n$ follows immediately from Propositions~\ref{th4bpro4***}-- \ref{th4bpro4} and \ref{th4bpro6}.

It remains to prove (\ref{asden2}). From Theorem~2.6 \cite{ChG:2013} and the formula (\ref{loc.4}) it follows immediately that
\begin{equation}\label{th4bpro5.4}
\sup_{x\in\mathbb R}|F_n(x+a_n)-\int_{-\infty}^x v_n(u)\,du|\le c(\mu)\,\frac{\varepsilon_{n2}}n, 
\end{equation}
where $\varepsilon_{n2}:=\eta_q(n)\frac{\beta_q(\mu)}{n^{(q-4)/2}}$ if $\beta_q(\mu):=\int_{\mathbb R}|u|^q\,\mu(du)<\infty,\,4\le q<5$. 
Here 
$$
\eta_q(n)=\inf_{0<\varepsilon\le 10^{-1/2}}g_q(\varepsilon),\quad\text{where}\quad 
g_q(\varepsilon)=\varepsilon^{q_*}+\frac{\varepsilon^{q_*-5}}{\beta_{5-q_*}(\mu)}
\int_{|u|>\varepsilon\sqrt n}|u|^{5-q_*}\,\mu(du)
$$
with $q_*=5-\min\{q,5\}$. It is easy to see that
$\eta_q(n)$ are the functions such that $\eta_q(n)\le 10^{1+3/2}$ and $\eta_q(n)\to 0$ monotonically as $n\to\infty$.
If $\beta_q(\mu)<\infty,\,q\ge 5$, then $\varepsilon_{n2}:=\frac{\beta_5(\mu)}{n^{1/2}}$. Therefore (\ref{asden2}) holds with 
the indicated sequence $\{\varepsilon_{n2}\}$.

Theorem~\ref{th7} is completely proved.

\begin{remark}\label{rem8.11}
 It is obvious that for the considered above $\varepsilon_{n2}$ we have the properties:
 $\frac{\varepsilon_{n2}}n=o(1/n^{(q-2)/2})$ if $\beta_q(\mu)<\infty, \,4\le q<5$, and $\frac{\varepsilon_{n2}}n=O(1/n^{3/2})$ if
 $\beta_5(\mu)<\infty$.
\end{remark}

\begin{remark}\label{rem8.12}
If $m_{2k}(\mu)<\infty$ for some $k=3,\dots$, then we choose in (\ref{Pas5.0})
$\delta_n=1/\pi$ and define $\varepsilon_{n1}:=c_1(\mu)(n\tau(\mathbb R\setminus[-\sqrt{n-1}/\pi,\sqrt{n-1}/\pi])+1/\sqrt n)$.
In this case $\varepsilon_{n1}=O(1/\sqrt n)$. Repeating the argument of Sections 5--8 we obtain the statement of Theorem~\ref{th7}
with such $\varepsilon_{n1}$ and $\varepsilon_{n2}$ described before.
\end{remark}

\section{Asymptotic expansion of $\int_{\mathbb R} |p_n(x)-p_w(x)|\,dx$} 

In this section we shall prove Corollary~\ref{corth7.1}. Indeed, by the estimate (\ref{asden2}), 
we have, for $n\ge n_1$, 
\begin{align}\label{exp1.1}
\int_{\mathbb R} |p_n(x)-p_w(x)|\,dx=\int_{I_n} |p_n(x)-p_w(x)|\,dx+c(\mu)\,\theta\,\frac{\varepsilon_{n1}+\varepsilon_{n2}}n.
\end{align} 
Using Theorems~\ref{th7},
we easily conclude that
\begin{align}\notag
&\int_{I_n} |p_n(x)-p_w(x)|\,dx=\int_{I_n} |v_n(x-a_n)-p_w(x)|\,dx\notag\\
&+\theta\int_{I_n-a_n} |\rho_{n1}(x)|\,dx
+\theta\int_{I_n-a_n} |\rho_{n2}(x)|\,dx\notag\\ 
&=\int_{[-2,2]} |(1-a_nx)p_w(x)-p_w(x+a_n)|\,dx+c(\mu)\theta\,n^{-1}+\theta\int_{I_n-a_n} |\rho_{n1}(x)|\,dx\notag\\
&=\frac{|a_n|}{2\pi}\int_{[-2,2]} |x|\frac{|3-x^2|}{\sqrt{4-x^2}}\,dx+c(\mu)\theta\,n^{-1}
+\theta\int_{I_n-a_n} |\rho_{n1}(x)|\,dx\notag\\
&=\frac{2|a_n|}{\pi}+c(\mu)\theta\,n^{-1}+\theta\int_{I_n-a_n} |\rho_{n1}(x)|\,dx.\notag
\end{align}
By (\ref{asden1}), we see that
\begin{equation}\notag
\int_{I_n^*-a_n} |\rho_{n1}(x)|\,dx\le c(\mu)\Big(\frac{\varepsilon_{n1}}{n}\Big)^{3/4}.
\end{equation}
and, by (\ref{asden1*}),
\begin{equation}\notag
\int_{(I_n-a_n)\setminus(I_n^*-a_n)} |\rho_{n1}(x)|\,dx\le c(\mu)\Big(\frac{\varepsilon_{n1}}{n}\Big)^{3/4}.
\end{equation}

Applying these relations to (\ref{exp1.1}) we get the expansion (\ref{2.9}).

\section{Asymptotic expansion of the free entropy} 

In this section we prove Corollaries~\ref{corth7.2}. First we find an asymptotic expansion of 
the logarithmic energy $E(\mu_n)$ of the measure $\mu_n$. Recall that (see~\cite{HiPe:2000})
\begin{align}\label{exp.1}
-E(\mu_n)&=\int\int_{\mathbb R^2}\log|x-y|\,\mu_n(dx)\mu_n(dy)=I_1(\mu_n)+I_2(\mu_n)\notag\\
&:=\int\int_{I_n\times I_n}\log|x-y|\,\mu_n(dx)\mu_n(dy)+
\int\int_{\mathbb R^2\setminus (I_n\times I_n)}\log|x-y|\,\mu_n(dx)\mu_n(dy).
\end{align}
Using (\ref{asden3}) and the equality $\int_{\mathbb R}u^2\,p_n(u)\,du=1$,  
we get the inequality
\begin{equation}\notag
\int_{\mathbb R}(1+|y|)^2p_n(y)^2\,dy\le c(\mu)\int_{\mathbb R}(1+|y|)^2p_n(y)\,dy\le c(\mu).
\end{equation}
Therefore we conclude with the help of the Cauchy-Bunyakovsky inequality that
\begin{align}\notag
\int_{\mathbb R}|\log|x-y||\,p_n(y)dy\le \Big(\int_{\mathbb R}\frac{(\log|x-y|)^2}{(1+|y|)^2}\,dy\Big)^{1/2}
\Big(\int_{\mathbb R}(1+|y|)^2p_n(y)^2\,dy\Big)^{1/2}\le c(\mu).
\end{align}
Recalling (\ref{asden2}), we obtain
\begin{equation}\label{exp.2}
|I_2(\mu_n)|\le 4\int_{\mathbb R\setminus I_n}p_n(x)\,\int_{\mathbb R}|\log|x-y||\,p_n(y)\,dy\,dx\le 
c(\mu)\int_{\mathbb R\setminus I_n}p_n(x)\,dx
\le c(\mu)\,\frac{\varepsilon_{n2}}n. 
\end{equation}
Now we note that
\begin{align}
I_1(\mu_n)&=\int\int_{I_n\times I_n}\log|x-y|\,p_n(x)p_n(y)\,dxdy=I_{11}(\mu_n)+2I_{12}(\mu_n)
+I_{13}(\mu_n)\notag\\
&:=\int_{I_n\times I_n}\log|x-y|\,v_n(x-a_n)v_n(y-a_n)\,dxdy\notag\\
&+2\int\int_{I_n\times I_n}\log|x-y|\,(p_n(x)-v_n(x-a_n))v_n(y-a_n)\,dxdy\notag\\
&+\int\int_{I_n\times I_n}\log|x-y|\,(p_n(x)-v_n(x-a_n))(p_n(y)-v_n(y-a_n))\,dxdy.\label{exp.2a}
\end{align}

Using the form of $v_n(x)$ we easily conclude that 
\begin{align}\label{exp.3}
I_{11}(\mu_n)
=\int\int_{\mathbb R^2}\log|x-y|\,v_n(x)v_n(y)\,dx\,dy
+c\theta \Big(\frac{\varepsilon_{n1}}n\Big)^{3/2}.
\end{align}
Recalling the definition of $v_n(x)$ we see that
\begin{align}\label{exp.4}
&\int\int_{\mathbb R^2}\log|x-y|\,v_n(x)v_n(y)\,dx\,dy=\tilde{I}_1(v_n)+\tilde{I}_2(v_n)+\tilde{I}_3(v_n)
+\tilde{I}_4(v_n)+c(\mu)\theta n^{-2}\notag\\
&:=\Big(1+\frac 12 d_n-a_n^2-\frac 1{n}\Big)^2 \int\int_{\mathbb R^2}\log|x-y|\,
p_w(e_nx)\,p_w(e_ny)\,dx\,dy\notag\\
&-2\Big(1+\frac 12 d_n-a_n^2-\frac 1{n}\Big)a_n\int\int_{\mathbb R^2}x\log|x-y|\,
p_w(e_nx)\,p_w(e_ny)\,dx\,dy\notag\\
&+a_n^2\int\int_{\mathbb R^2}xy\log|x-y|\,p_w(e_nx)\,
p_w(e_ny)\,dx\,dy\notag\\
&-2\Big(b_n-a_n^2-\frac 1{n}\Big)\int\int_{\mathbb R^2}x^2\log|x-y|\,
p_w(e_nx)\,p_w(e_ny)\,dx\,dy+c(\mu)\theta n^{-2}. 
\end{align}

In view of $E(\mu_w)=1/4$ and $e_n^{-2}=1-d_n+2b_n+c(\mu)\theta n^{-2},\,
\log e_n=\frac{d_n}2-b_n+c(\mu)\theta n^{-2}$,
note that
\begin{align}\label{exp.5}
\tilde{I}_1(v_n)&=\Big(1+\frac 12 d_n-a_n^2-\frac 1{n}\Big)^2 e_n^{-2}
\Big(\int\int_{\mathbb R^2}\log|x-y|
p_w(x)p_w(y)\,dx\,dy-\log e_n\Big)\notag\\
&=-E(\mu_w)+\frac{a_n^2}2+c(\mu)\theta n^{-2}. 
\end{align}
Since the function $\int_{\mathbb R}\log|x-y|p_w(y)\,dy$ is even, we see that
$\tilde{I}_2(v_n)=0$. In order to calculate $\tilde{I}_3(v_n)$ we easily deduce that
\begin{equation}\label{exp.5a}
\int_{-2}^2 up_w(u)\log|x-u|\,du=-x+x^3/6,\quad x\in[-2,2]. 
\end{equation}
Therefore we obtain
\begin{equation}\label{exp.6}
\tilde{I}_3(v_n)=-a_n^2\int_{-2}^2xp_w(x)(x-x^3/6)\,dx+c(\mu)\theta n^{-2}=-\frac 23a_n^2+c(\mu)\theta n^{-2}. 
\end{equation}

Using the following well-known formula (see~\cite{HiPe:2000}, p. 197)
\begin{equation}\label{exp.6a}
\int_{-2}^2 p_w(u)\log|x-u|\,du=\frac{x^2}4-\frac 12,\quad x\in [-2,2],
\end{equation}
we deduce
\begin{equation}\label{exp.6b}
\int\int_{\mathbb R^2}x^2\log|x-y|\,
p_w(x)\,p_w(y)\,dx\,dy=\int_{-2}^2x^2\Big(\frac{x^2}4-\frac 12\Big)p_w(x)\,dx=0. 
\end{equation}
Therefore
\begin{equation}\label{exp.7}
\tilde{I}_4(v_n)=c(\mu)\theta\,n^{-2}. 
\end{equation}

By (\ref{exp.3})--(\ref{exp.7}), we arrive at the formula
\begin{equation}\label{exp.7aa}
I_{11}(\mu_n)=-E(\mu_w)-\frac 16a_n^2+c\theta \Big(\Big(\frac{\varepsilon_{n1}}n\Big)^{3/2}+\frac 1{n^2}\Big). 
\end{equation}

Now we note, using (\ref{exp.5a}), (\ref{exp.6a}), (\ref{exp.6b}), and the H\"older inequality that, for $x\in I_n-a_n$
and for any fixed positive $\delta$,
\begin{align}
&\int_{I_n-a_n}\log|x-y|\,v_n(y)\,dy\notag\\
&=\int_{\mathbb R}\log|x-y|\big(1+\frac 12 d_n-a_n^2-\frac 1n-a_ny-(b_n-a_n^2-\frac 1n)y^2\big)
\,p_w(e_ny)\,dy\notag\\
&+c(\mu,\delta)\theta\,\Big(\frac{\varepsilon_{n1}}n\Big)^{\frac 32-\delta}
=\frac{x^2}4-\frac 12-\log e_n+a_n\big(x-\frac{x^3}6\big)+c(\mu,\delta)\theta\,\Big(\frac{\varepsilon_{n1}}n\Big)^{\frac 32-\delta}, \notag
\end{align}
where $c(\mu,\delta)$ are positive constants depending on $\mu$ and $\delta$ only.

On the other hand, by (\ref{th4bpro5.4}), we see that
\begin{equation}\notag
\Big|\int_{I_n}x^k(p_n(x+a_n)-v_n(x))\,dx\Big|\le  c(\mu)\,\frac{\varepsilon_{n2}}n, \quad k=0,1,2,3.\notag 
\end{equation}
The last two relations give us finally
\begin{equation}
|I_{12}(\mu_n )|\le c(\mu,\delta)\,\Big(\frac{\varepsilon_{n1}}n\Big)^{\frac 32-\delta}+c(\mu)\,\frac{\varepsilon_{n2}}n.\label{exp.7a}
\end{equation}

It remains to estimate $I_{13}(\mu_n )$. Write
\begin{align}
I_{13}(\mu_n )&= I_{13,1}(\mu_n )+2I_{13,2}(\mu_n)
+I_{13,3}(\mu_n)\notag\\
&:=\Big(\int\int_{I_n^*\times I_n^*}+2\int\int_{I_n^*\times (I_n\setminus I_n^*)}
+\int\int_{(I_n\setminus I_n^*)\times (I_n\setminus I_n^*)}\Big)\notag\\
&\log|x-y|\,(p_n(x)-v_n(x-a_n))(p_n(y)-v_n(y-a_n))\,dxdy. \label{exp.8}
\end{align}
In order to estimate $I_{13,1}(\mu_n)$ we use the upper bound 
\begin{align}
\tilde{I}_{13,1}&:=\int_{I_n^*}|\log|x-y||\Big(
|\rho_{n1}(y-a_n)|+\rho_{n2}(y-a_n)\Big)\,dy\notag\\
&\le c(\mu)\frac{\varepsilon_{n1}}n\int_{I_n^*}|\log|x-y||\frac{dy}{(4-(e_n(y-a_n))^2)^{3/2}}+\int_{I_n^*}|\log|x-y||\rho_{n2}(y)\,dy.\notag
\end{align}
Let $p,q> 1$ and $\frac 1p+\frac 1q=1$. We assume that $q$ is closed to $1$, i.e., $1<q\le 1.01$. By H\"older inequality, for $x\in I_n^*$,
\begin{align}
 &\int_{I_n^*}|\log|x-y||\frac{dy}{(4-(e_n(y-a_n))^2)^{3/2}}\notag\\
 &\le \Big(\int_{I_n^*}|\log|x-y||^{p}\,dy\Big)^{1/p}
 \Big(\int_{I_n^*}\frac{dy}{(4-(e_n(y-a_n))^2)^{3q/2}}\,dy\Big)^{1/q}\le c(\mu,q)\Big(\sqrt{\frac{n}{\varepsilon_{n1}}}\Big)^{\frac 32-\frac 1q}\notag
\end{align}
and, by (\ref{asden1**}), 
\begin{align}
\int_{I_n^*}|\log|x-y||\rho_{n2}(y-a_n)\,dy\le \Big(\int_{I_n^*}|\log|x-y||^{p}\,dy\Big)^{1/p}\Big(\int_{I_n^*}\rho_{n2}(y-a_n)^q\,dy\Big)^{1/q}
\le \frac {c(\mu,q)}{n^{2/q}},\notag
\end{align}
where $c(\mu,q)>0$ is a constant depended on $\mu$ and $q$ only. 
From the two last upper bounds we get
\begin{equation}\label{fexp.1}
\tilde{I}_{13,1}\le c(\mu,q)\Big(\Big(\frac{\varepsilon_{n1}}n\Big)^{\frac 14+\frac 1{2q}}+\frac 1{n^{2/q}}\Big)
\le c(\mu,q)\Big(\frac{\varepsilon_{n1}}n\Big)^{\frac 14+\frac 1{2q}}.
\end{equation}
On the other hand it is easy to see that
\begin{align}
 \int_{I_n^*}\Big(
 |\rho_{n1}(y-a_n)|+\rho_{n2}(y-a_n)\Big)\,dy
 \le c(\mu)\Big(\Big(\frac{\varepsilon_{n1}}n\Big)^{\frac 34}+\frac 1{n^{2}}\Big)
 \le c(\mu)\Big(\frac{\varepsilon_{n1}}n\Big)^{\frac 34}.\label{fexp.2}
\end{align}
We conclude from ({\ref{fexp.1}) and ({\ref{fexp.2}) that
\begin{equation}\label{fexp.3}
 |I_{13,1}(\mu_n )|\le c(\mu,q)\Big(\frac{\varepsilon_{n1}}n\Big)^{1+\frac 1{2q}}.
\end{equation}

Now introduce the quantity
\begin{equation}\notag
\tilde{I}_{13,2}:=\int_{I_n\setminus I_n^*}|\log|x-y||\Big(
|\rho_{n1}(y-a_n)|+\rho_{n2}(y-a_n)\Big)\,dy.
\end{equation}
As above we obtain, for $x\in I_n$,
\begin{align}\notag
 &\int_{I_n\setminus I_n^*}|\log|x-y||\frac{dy}{(4-(e_n(y-a_n))^2)^{1/2}}\notag\\
 &\le \Big(\int_{I_n\setminus I_n^*}|\log|x-y||^{p}\,dy\Big)^{1/p}
 \Big(\int_{I_n\setminus I_n^*}\frac{dy}{(4-(e_n(y-a_n))^2)^{q/2}}\,dy\Big)^{1/q}
 \le c(\mu,q)\Big(\frac {\varepsilon_{n1}}n\Big)^{\frac 1{2q}-\frac {1}4}.\notag
\end{align}
Therefore we have
\begin{equation}\label{fexp.4}
\tilde{I}_{13,2}\le  c(\mu,q)
\Big(\frac{\varepsilon_{n1}}{n}\Big)^{\frac 1{2q}+\frac {1}4},\quad x\in I_n.
\end{equation}
We deduce from (\ref{fexp.2}) and (\ref{fexp.4}) 
\begin{equation}\label{fexp.5}
 |I_{13,2}(\mu_n )|\le c(\mu,q)
 \Big(\frac{\varepsilon_{n1}}{n}\Big)^{\frac 1{2q}+1}.
\end{equation} 

It remains to estimate $I_{13,3}(\mu_n )$. It is easy to verify that
\begin{equation}
 \int_{I_n\setminus I_n^*}\Big(|\rho_{n1}(y-a_n)|+\rho_{n2}(y-a_n)\Big)\,dy\le c(\mu,q)
\Big(\frac{\varepsilon_{n1}}n\Big)^{\frac 34}.\label{fexp.6}
\end{equation}
 
We obtain from (\ref{fexp.4}) and (\ref{fexp.6}) that
\begin{equation}\label{fexp.7}
 |I_{13,3}(\mu_n )|\le c(\mu,q)\Big(\frac{\varepsilon_{n1}}{n}\Big)^{\frac 1{2q}+1}.
\end{equation}
It remains to note that the upper bound
\begin{equation}\label{fexp.8}
 |I_{13}(\mu_n )|\le c(\mu,q)\Big(\frac{\varepsilon_{n1}}{n}\Big)^{\frac 1{2q}+1}.
\end{equation}
follows immediately from (\ref{fexp.3}), (\ref{fexp.5}), and (\ref{fexp.7}).

In view of (\ref{exp.1}), (\ref{exp.2}), (\ref{exp.7aa}), (\ref{exp.7a}) and (\ref{fexp.8}) we get
\begin{equation}\label{fexp.9}
-E(\mu_n)=-E(\mu_w)-\frac 16a_n^2+c(\mu,q)\theta\,\Big(\Big(\frac{\varepsilon_{n1}}{n}\Big)^{\frac 1{2q}+1}+
\frac {\varepsilon_{n2}}n \Big). 
\end{equation}
The assertion of Corollary~\ref{corth7.2} follows from this relation.

\section{Asymptotic expansion of the free Fisher information}

Now let us prove Corollary~\ref{corth7.3}. We shall show that the free Fisher information of the measure $\mu_n$ 
has the form (\ref{2.11}).
Denote 
\begin{align}\label{exp.9}
\Phi(\mu_n)=\Phi_1(\mu_n)+\Phi_2(\mu_n):=\frac{4\pi^2}3\int_{I_n}p_n(x)^3\,dx+
\frac{4\pi^2}3\int_{\mathbb R\setminus I_n}p_n(x)^3\,dx.
\end{align}
As before we see that, by (\ref{asden2}),
\begin{equation}\label{exp.10}
\Phi_2(\mu_n)\le c(\mu)\int_{\mathbb R\setminus I_n}p_n(x)\,dx\le c(\mu)\frac {\varepsilon_{n2}}n.
\end{equation}
On the other hand, by (\ref{asden}), we have 
\begin{equation}\label{exp.11}
\Phi_1(\mu_n)=\frac{4\pi^2}3\int_{I_n-a_n}v_n(x)^3\,dx+\frac{4\pi^2}3
\int_{I_n-a_n}\sum_{\scriptstyle k+l=3\atop\scriptstyle 0\le k\le 2,\,l\ge 1}v_n(x)^k(\rho_{n1}(x)+\rho_{n2}(x))^l\,dx. 
\end{equation}
In the sequal we consider nonnegative entire numbers $0\le k\le 2,\,l\ge 1$ and $k+l=3$ only.

We see, by (\ref{asden1**}), that
\begin{equation}\label{exp.12}
 \int_{I_n-a_n}|v_n(x)|^k|\rho_{n2}(x)|^l\,dx\le c(\mu)\frac 1{n^2}.
\end{equation}

Now we note, by (\ref{asden1}), that 
\begin{align}\label{exp.13}
 \int_{I_n^*- a_n}|v_n(x)|^k|\rho_{n1}(x)|^l\,dx&\le c(\mu)\Big(\frac {\varepsilon_{n1}}{n}\Big)^l\int_{I_n^*- a_n}
 \frac{dx}{(4-(e_nx)^2)^{(3l-k)/2}}\notag\\
& \le c(\mu)\Big(\frac {\varepsilon_{n1}}{n}\Big)^l\Big(\frac n{\varepsilon_{n1}}\Big)^{\frac{3l-k-2}4}
=c(\mu)\Big(\frac {\varepsilon_{n1}}{n}\Big)^{\frac 54}.
\end{align}
Furthermore, we have, using the bound (\ref{asden1*}),
\begin{align}\label{exp.14}
&\int_{(I_n-a_n)\setminus(I_n^*-a_n)}|v_n(x)|^k|\rho_{n1}(x)|^l\,dx\notag\\
&\le c(\mu)\int_{(I_n-a_n)\setminus(I_n^*-a_n)}|v_n(x)|^k\Big(\frac {\varepsilon_{n1}}{n(4-(e_nx)^2)}\Big)^{l/2}\,dx
\le c(\mu)\frac {\varepsilon_{n1}}n.
\end{align}
From (\ref{exp.13}) and (\ref{exp.14}) it follows
\begin{equation}\label{exp.15}
 \int_{I_n-a_n}|v_n(x)|^k|\rho_{n1}(x)|^l\,dx\le c(\mu)\,\frac {\varepsilon_{n1}}{n}.
\end{equation}
Applying (\ref{exp.12}) and (\ref{exp.15}) to (\ref{exp.11}) we obtain
\begin{equation}\label{exp.16}
\Phi_1(\mu_n)=\frac{4\pi^2}3\int_{I_n-a_n}v_n(x)^3\,dx+c(\mu)\theta\,\frac{\varepsilon_{n1}}n.
\end{equation}
It is easy to see that the integral on the right hand-side of (\ref{exp.16}) is equal to
\begin{align}
\Big(1&+3\Big(\frac 12 d_n-a_n^2-\frac 1n\Big)\Big)e_n^{-1}\int_{\mathbb R}p_w(x)^3\,dx
-3\Big(b_n-a_n^2-\frac 1n\Big)e_n^{-3}\int_{\mathbb R}x^2p_w(x)^3\,dx \notag\\
&+3a_n^2e_n^{-3}\int_{\mathbb R}x^2p_w(x)^3\,dx+c(\mu)\theta n^{-5/2}=\frac 3{4\pi^2}(1+a_n^2)
+c(\mu)\theta \, n^{-2}.\notag
\end{align}
Therefore we finally conclude by (\ref{exp.9})--(\ref{exp.11}) that
\begin{equation}\notag
\Phi(\mu_n)=1+a_n^2+\theta c(\mu)\frac{\varepsilon_{n1}+\varepsilon_{n2}}n=\Phi(\mu_w)+a_n^2
+c(\mu)\theta \,\frac{\varepsilon_{n1}+\varepsilon_{n2}}n.
\end{equation}
Thus, Corollary~\ref{corth7.3} is proved.



\end{document}